\documentclass[12pt]{amsart}
\usepackage{amscd,amssymb,latexsym}
\usepackage{fullpage}
\usepackage{fullpage}

\renewcommand{\thesubsection}{\thesection.\Alph{subsection}}

\newcommand{\Mdef}[2]{\newcommand{#1}{\relax \ifmmode #2 \else $#2$\fi}}

\newcommand{\cok}{\mathrm{cok}}

\newcommand{\sm }{\wedge}

\newcommand{\tensor}{\otimes}

\newcommand{\map}{\mathrm{map}}
\newcommand{\Hom}{\mathrm{Hom}}
\newcommand{\Tor}{\mathrm{Tor}}
\newcommand{\Ext}{\mathrm{Ext}}
\Mdef{\bhom}{\mathbf{\hat{H}om}}

\Mdef{\Mod}{\mathrm{mod}}

\newcommand{\st}{\; | \;}

\newtheorem{thm}{Theorem}[section]
\newtheorem{lemma}[thm]{Lemma}
\newtheorem{prop}[thm]{Proposition}
\newtheorem{cor}[thm]{Corollary}

\theoremstyle{definition}

\newtheorem{defn}[thm]{Definition}

\newtheorem{construction}[thm]{Construction}

\newtheorem{context}[thm]{Context}

\newtheorem{example}[thm]{Example}
\newtheorem{examples}[thm]{Examples}
\newtheorem{remark}[thm]{Remark}

\newcommand{\qqed}{\qed \\[1ex]}
\renewenvironment{proof}[1][\hspace*{-.8ex}]{\noindent {\bf Proof #1:\;}}{\qqed}

\Mdef{\PH} {\Phi^H}
\Mdef{\PK} {\Phi^K}
\Mdef{\PL} {\Phi^L}
\Mdef{\PT} {\Phi^{\T}}

\Mdef{\ef}{E{\cF}_+}
\Mdef{\etf}{\tilde{E}{\cF}}
\Mdef{\eg}{E{G}_+}
\Mdef{\etg}{\tilde{E}{G}}
\Mdef{\tf}{\T / \! \!  / {\cF}_+}

\Mdef{\infl}{\mathrm{inf}}
\Mdef{\defl}{\mathrm{def}}
\Mdef{\res}{\mathrm{res}}
\Mdef{\ind}{\mathrm{inf}}

\Mdef{\univ}{\mathcal{U}}

\newcommand{\CP}{\mathbb{C} P}

\Mdef{\Fp}{\mathbb{F}_p}
\Mdef{\Zpinfty}{\Z /p^{\infty}}
\Mdef{\Zpadic}{\Z_p^{\wedge}}

\newcommand{\bi}{\begin{itemize}}
\newcommand{\be}{\begin{enumerate}}
\newcommand{\bc}{\begin{center}}
\newcommand{\bd}{\begin{description}}
\newcommand{\ei}{\end{itemize}}
\newcommand{\ee}{\end{enumerate}}
\newcommand{\ec}{\end{center}}
\newcommand{\ed}{\end{description}}

\newcommand{\RLadjunction}[4]{
\diagram
#1:#2 \rrto<-0.7ex> &&
#3  \llto<-0.7ex> :#4 
\enddiagram}
\newcommand{\adjunction}[4]{
\diagram
#1:#2 \rrto<0.7ex> &&
#3  \llto<0.7ex> :#4 
\enddiagram}
\newcommand{\lra}{\longrightarrow}
\newcommand{\lla}{\longleftarrow}

\newcommand{\iso}{\cong}

\newcommand{\spec}{\mathrm{spec}}

\Mdef{\we}{\mathbf{we}}
\Mdef{\fib}{\mathbf{fib}}
\Mdef{\cof}{\mathbf{cof}}
\Mdef{\BI}{\mathcal{BI}}

\newcommand{\ann}{\mathrm{ann}}
\newcommand{\fibre}{\mathrm{fibre}}

\newcommand{\ilim}{\mathop{ \mathop{\mathrm{lim}} \limits_\leftarrow} \nolimits}
\newcommand{\colim}{\mathop{  \mathop{\mathrm {lim}} \limits_\rightarrow} \nolimits}
\newcommand{\holim}{\mathop{ \mathop{\mathrm {holim}} \limits_\leftarrow} \nolimits}
\newcommand{\hocolim}{\mathop{  \mathop{\mathrm {holim}}\limits_\rightarrow} \nolimits}

\newcommand{\sfD}{\mathsf{D}}
\newcommand{\Ho}{\mathrm{Ho}}

\Mdef{\A}{\mathbb{A}}
\Mdef{\B}{\mathbb{B}}
\Mdef{\C}{\mathbb{C}}
\Mdef{\D}{\mathbb{D}}
\Mdef{\E}{\mathbb{E}}
\Mdef{\T}{\mathbb{T}}
\Mdef{\F}{\mathbb{F}}
\Mdef{\G}{\mathbb{G}}
\Mdef{\I}{\mathbb{I}}
\Mdef{\N}{\mathbb{N}}
\Mdef{\Q}{\mathbb{Q}}
\Mdef{\R}{\mathbb{R}}
\Mdef{\bbS}{\mathbb{S}}
\Mdef{\Z}{\mathbb{Z}}

\Mdef{\bA}{\mathbb{A}}
\Mdef{\bB}{\mathbb{B}}
\Mdef{\bC}{\mathbb{C}}
\Mdef{\bD}{\mathbb{D}}
\Mdef{\bE}{\mathbb{E}}
\Mdef{\bF}{\mathbb{F}}
\Mdef{\bG}{\mathbb{G}}
\Mdef{\bH}{\mathbb{H}}
\Mdef{\bI}{\mathbb{I}}
\Mdef{\bJ}{\mathbb{J}}
\Mdef{\bK}{\mathbb{K}}
\Mdef{\bL}{\mathbb{L}}
\Mdef{\bM}{\mathbb{M}}
\Mdef{\bN}{\mathbb{N}}
\Mdef{\bO}{\mathbb{O}}
\Mdef{\bP}{\mathbb{P}}
\Mdef{\bQ}{\mathbb{Q}}
\Mdef{\bR}{\mathbb{R}}
\Mdef{\bS}{\mathbb{S}}
\Mdef{\bT}{\mathbb{T}}
\Mdef{\bU}{\mathbb{U}}
\Mdef{\bV}{\mathbb{V}}
\Mdef{\bW}{\mathbb{W}}
\Mdef{\bX}{\mathbb{X}}
\Mdef{\bY}{\mathbb{Y}}
\Mdef{\bZ}{\mathbb{Z}}

\Mdef{\cA}{\mathcal{A}}
\Mdef{\cB}{\mathcal{B}}
\Mdef{\cC}{\mathcal{C}}
\Mdef{\mcD}{\mathcal{D}} 
\Mdef{\cE}{\mathcal{E}}
\Mdef{\cF}{\mathcal{F}}
\Mdef{\cG}{\mathcal{G}}
\Mdef{\mcH}{\mathcal{H}} 
\Mdef{\cI}{\mathcal{I}}
\Mdef{\cJ}{\mathcal{J}}
\Mdef{\cK}{\mathcal{K}}
\Mdef{\mcL}{\mathcal{L}}

\Mdef{\cM}{\mathcal{M}}
\Mdef{\cN}{\mathcal{N}}
\Mdef{\cO}{\mathcal{O}}
\Mdef{\cP}{\mathcal{P}}
\Mdef{\cQ}{\mathcal{Q}}
\Mdef{\mcR}{\mathcal{R}}
\Mdef{\cS}{\mathcal{S}}
\Mdef{\cT}{\mathcal{T}}
\Mdef{\cU}{\mathcal{U}}
\Mdef{\cV}{\mathcal{V}}
\Mdef{\cW}{\mathcal{W}}
\Mdef{\cX}{\mathcal{X}}
\Mdef{\cY}{\mathcal{Y}}
\Mdef{\cZ}{\mathcal{Z}}

\Mdef{\tA}{\tilde{A}}
\Mdef{\tB}{\tilde{B}}
\Mdef{\tC}{\tilde{C}}
\Mdef{\tE}{\tilde{E}}
\Mdef{\tH}{\tilde{H}}
\Mdef{\tK}{\tilde{K}}
\Mdef{\tL}{\tilde{L}}
\Mdef{\tM}{\tilde{M}}
\Mdef{\tN}{\tilde{N}}
\Mdef{\tP}{\tilde{P}}

\Mdef{\Ab}{\overline{A}}
\Mdef{\Bb}{\overline{B}}
\Mdef{\Cb}{\overline{C}}
\Mdef{\Db}{\overline{D}}
\Mdef{\Hb}{\overline{H}}
\Mdef{\Ib}{\overline{I}}
\Mdef{\Kb}{\overline{K}}
\Mdef{\Lb}{\overline{L}}
\Mdef{\Mb}{\overline{M}}
\Mdef{\Nb}{\overline{N}}
\Mdef{\Qb}{\overline{Q}}
\Mdef{\Tb}{\overline{T}}

\Mdef{\qb}{\overline{q}}
\Mdef{\rb}{\overline{r}}
\Mdef{\tb}{\overline{t}}
\Mdef{\ub}{\overline{u}}
\Mdef{\vb}{\overline{v}}
\Mdef{\hc}{\hat{c}}
\Mdef{\he}{\hat{e}}
\Mdef{\hf}{\hat{f}}
\Mdef{\hA}{\hat{A}}
\Mdef{\hH}{\hat{H}}
\Mdef{\hJ}{\hat{J}}
\Mdef{\hM}{\hat{M}}
\Mdef{\hP}{\hat{P}}
\Mdef{\hQ}{\hat{Q}}

\Mdef{\bolda}{\mathbf{a}}
\Mdef{\boldb}{\mathbf{b}}
\Mdef{\boldD}{\mathbf{D}}

\Mdef{\fm}{\frak{m}}
\Mdef{\fX}{\frak{X}}

\Mdef{\eps}{\epsilon}

\input{xypic}

\newcommand{\tand}{\mbox{ and }}
\newcommand{\sE}{\cE}
\newcommand{\epz}{\varepsilon}
\newcommand{\sset}[1]{#1}
\newcommand{\SI}{\Sigma}
\newcommand{\htp}{\simeq}

\newcommand{\FF}{\Bbb{F}}

\newcommand{\darrow}{\longrightarrow}

\newcommand{\AL}{\mbox{\boldmath $\alpha$}} 
\newcommand{\FK}[1]{K^{\bullet}_{\infty}( #1)} 
\newcommand{\HFK}[1]{K_{\infty}( #1)} 
\newcommand{\PPK}[1]{PK_{\infty}^{\bullet}( #1)} 
\newcommand{\FC}[1]{\check{C}^{\bullet}( #1)} 
\newcommand{\PC}[1]{P\check{C}^{\bullet}( #1)}

\newcommand{\CJI}{\check{C}(I)}
\newcommand{\CIM}{M[I^{-1}]}

\newcommand{\boldR}{R}
\newcommand{\boldk}{k}
\newcommand{\HombR}{\Hom_{\boldR}}
\newcommand{\HomR}{\Hom_{R}}

\newcommand{\kvee}{k^{\vee}}

\newcommand{\tensorcE}{\tensor_\cE}
\newcommand{\tensorR}{\tensor_R}
\newcommand{\modcE}{\mathrm{mod}\!-\!\cE}
\newcommand{\bRmod}{\boldR\!-\!\mathrm{mod}}

\newcommand{\cell}{\mathrm{Cell}}

\newcommand{\Gk}{\Gamma_k}
\newcommand{\Lk}{\Lambda^k}
\newcommand{\Rmod}{\mbox{$R$-mod}}

\newcommand{\End}{\mathrm{End}}

\newcommand{\HomE}{\mathrm{Hom}_{\cE}}

\newcommand{\Homk}{\mathrm{Hom}_k}
\newcommand{\sdr}{\rtimes}

\newcommand{\tensorE}{\otimes_{\cE}}

\newcommand{\Kinfty}{K_{\infty}}
\newcommand{\Ep}{E'}
\newcommand{\ksharp}{k^{\#}}
\newcommand{\depth}{\mathrm{depth}}
\newcommand{\AR}{R}

\newcommand{\un}{\mathrm{un}}

\newcommand{\Kbar}{\overline{K}}
\newcommand{\Fbar}{\overline{F}}

\newcommand{\builds}{\vdash}
\newcommand{\finbuilds}{\models}
\newcommand{\finbuiltby}{=\!\!|}

\newcommand{\shift}{\mathrm{shift}}
\newcommand{\Ftwo}{{\mathbb{F}}_{2}}

\newcommand{\Mt}{\tilde{M}}
\newcommand{\fn}{\mathfrak{n}}
\newcommand{\fa}{\mathfrak{a}}
\newcommand{\BC}{BC}
\newcommand{\ABC}{ABC}

\newcommand{\bbF}{\mathbb{F}}
\newcommand{\uFp}{\underline{\mathbb{F}}_p}
\newcommand{\lc}{H_{\fm}}

\title{Homotopy invariant commutative algebra over fields}
\author{J.P.C. Greenlees}
\address{The University of Sheffield, Sheffield, S3 7RH\, UK} 
\email{j.greenlees@sheffield.ac.uk}
\begin{document}

\baselineskip=15pt

\newcommand{\blankbox}[2]{%
  \parbox{\columnwidth}{\centering
    \setlength{\fboxsep}{0pt}%
    \fbox{\raisebox{0pt}[#2]{\hspace{#1}}}%
  }%
}

\begin{abstract}
This article grew out of lectures given in the programme `Interactions
between Representation Theory, Algebraic Topology and Commutative
Algebra' (IRTATCA) at the CRM (Barcelona) in Spring 2015. They give
some basic homotopy invariant definitions in commutative algebra
and illustrate their interest by giving a number of examples. 
\end{abstract}

\thanks{I am grateful to the CRM and the Simons Foundation for the opportunity to spend time
in the stimulating environment at the CRM during the IRTATCA programme, and to give these lectures. }
\maketitle

\tableofcontents

\part{Context}
\section{Introduction}

\subsection{The lectures}
The purpose of these lectures is to illustrate how powerful it can be
to formulate ideas of commutative algebra in a homotopy invariant
form.  The point is  that it shows the concepts are robust, in the sense that they are 
invariant under deformations. They can then be applied to derived
categories of rings, but of course derived categories of rings are
just one example of a homotopy category of a stable model category. 
In short, these ideas are powerful in classical algebra, in representation theory of groups, in
classical algebraic topology and elsewhere. 
We will focus on the fact that this includes many striking new examples from topology. 

Indeed, our principal example will be that arising from  a topological
space $X$, which is to say we will work over  the ring $R=C^*(X;k)$ of
cochains on a space $X$ with coefficients in a field $k$. Of course
this is a rather old idea, applied in rational homotopy theory to give
a very rich theory and many interesting examples. The key is that $R$
needs to be commutative: classically, one needs  to work over the rationals so that there is a commutative model for cochains (coming from piecewise 
polynomial differential forms). It is well known that there is no natural model for $C^*(X;k)$ which is a commutative differential graded algebra (commutative DGA)  if 
$k$ is not of characteristic 0, since the Steenrod operations (which
do not usually vanish) are built out of the
non-commutativity. However, there is a way round this, since we can
relax our requirements for what we consider a model. Instead of
requiring DGAs,  we permit $R$ to be a commutative ring spectrum in
the sense of homotopy theory (see Section \ref{sec:constructspectra}
for a brief introduction). There is then a commutative 
model for cochains; we will continue to use the notation $C^*(X;k)$ even though it is now a ring spectrum rather than 
a DGA.

This may be our principal example, but actually it is a very
restricted one. It is like restricting graded commutative rings to be
negatively graded $k$-algebras. We will hint at a number of  examples coming from derived algebraic geometry. 

The main point of the lectures is to give some interesting examples,
and the main source of examples will be modular representation theory
and group cohomology. In fact, the focus will be on examples picked out by 
requiring the ring spectrum to be well behaved in terms of commutative
algebra: we look for regular, complete intersection or Gorenstein ring
spectra. As one expects from commutative algebra, regular ring spectra
are very restricted and can sometimes even be classified. One also
expects to be able to parametrise complete intersections in some
sense, although our present account just mentions a few isolated examples.
Finally, just as is the case for conventional commutative  rings, the class of Gorenstein ring spectra has proved
ubiquitous. We shall see examples from representation theory, from chromatic stable homotopy theory and from rational homotopy theory, and one of the distinctive contributions is to emphasize duality. 

\subsection{54321} It may be worth bearing in mind the following
correspondences, which suggest useful ways to think about the material
here. 

\vspace{1ex}
{\bf The pentagon:} We may summarize  sources  examples and
the relationship between them in a diagram. The emphasis here will be
on commutative algebra and spaces. 
$$
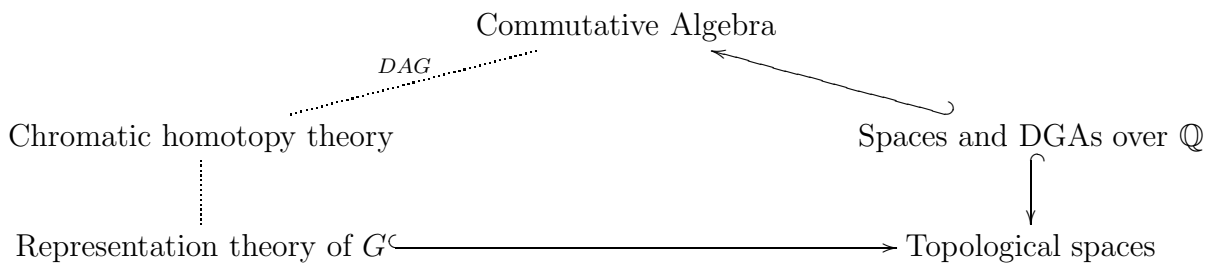
\diagram
&\ar@{.}[dl]_{DAG}\mbox{Commutative Algebra}&\\
\mbox{Chromatic homotopy theory}\ar@{.}[d]&&\mbox{Spaces and DGAs over $\Q$}\ar@{_{(}->}[ul]\ar@{^{(}->}[d]\\
\mbox{Representation theory of $G$}\ar@{^{(}->}[rr]
&&\mbox{Topological spaces}
\enddiagram$$

\vspace{1ex}
{\bf The square:} It is worth bearing in mind the dual points of view
of geometric objects and rings of functions on them, and the way the
global picture is built up out of points.  Our present
account focuses on the local rings at  the bottom right, but it can
be illuminating to consider the wider context.  
$$
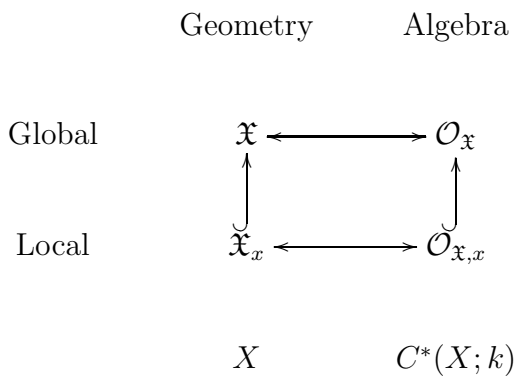
\diagram
&\mbox{Geometry}&\mbox{Algebra}\\
\mbox{Global}& \fX \ar@{<->}[r] &\cO_{\fX}\\
\mbox{Local}& \fX_x \ar@{<->}[r] \ar@{^{(}->}[u]
&\cO_{\fX,x} \ar@{^{(}->}[u]\\
&X&C^*(X;k)
\enddiagram$$

\vspace{1ex}
{\bf The three classes:} our focus is on the  following hierarchy of local
rings. 
$$\diagram
\mbox{Regular}\ar@{^{(}->}[r]&\mbox{Complete Intersection}\ar@{^{(}->}[r]&\mbox{Gorenstein}
\enddiagram$$

\vspace{1ex}
{\bf The dual pair:} even though our focus is on commutative rings, we
make constant use of the freedom to use Morita theory to transpose
properties into the non-commutative context. 
$$\diagram
\mbox{Commutative}\ar@{<->}[r]&\mbox{Non-commutative}
\enddiagram$$

\vspace{1ex}

\noindent
It is the purpose of the present article to give an account of these
objects and 
points of view  as a
\begin{center}
{\bf unified whole.} 
\end{center}

\subsection{Organization}
The  article is organized into five parts. In Part 1 we introduce
the context of ring spectra and modules over them, together with our
three main classes of examples (commutative rings, spaces and
representation theory) and the basic finiteness conditions. 

In Part 2 (Sections \ref{sec:Morita} to \ref{sec:exterior})
we describe the basic Morita equivalences between commutative and
non-commutative contexts, discuss our fundamental finiteness condition 
and illustrate its power by considering the richness of exterior algebras. 

We then turn to our hierarchy of local rings. In Part 3 (Sections
\ref{sec:regular} and \ref{sec:normalizable}) we consider
regular ring spectra; the principle of considering definitions in three
styles (structural, growth or module theoretic) recurs later, and the 
good behaviour of modules over regular rings is used later to give
a definition of `finitely generated' modules for many ring spectra.

 In Part 4 (Sections \ref{sec:ci} to \ref{sec:zhyperisghyper}) 
we consider hypersurface ring spectra. We spend a little time
recalling the  classical algebraic theory before introducing the
definitions for ring spectra, illustrating it with examples and explaining how 
the different versions are related.

Finally in Part 5 (Sections \ref{sec:Gor} to \ref{sec:Goregs}) we
consider Gorenstein ring spectra, and the duality properties they
often enjoy. We describe the basic tools for Gorenstein ring spectra
and illustrate them by explaining many known examples (especially
those from representation theory and rational homotopy theory). 
The rich and  attractive dualities that appear make this a highlight. 

A series of appendices recalls notation and basic properties of 
local cohomology and derived completion.

Each part starts with a  more detailed description of its contents.

\subsection{Sources}
Much of this account comes from joint work \cite{DGI, tec, cazci,
  pzci, qzci, THHGor} and other conversations with D.J.Benson,
W.G.Dwyer, K.Hess, S.B. Iyengar and  S. Shamir, and I would like to thank them for their collaboration and  inspiration. 

We summarize some relevant background, but commutative algebraists may find the articles \cite{spectra} 
and  \cite{firststeps} from the 2004 Chicago Workshop useful. General
background in topology may come from \cite{MayConcise} and in conventional commutative algebra
from \cite{Matsumura, BrunsHerzog}.

It is worth emphasizing that this is a vast area, and these lectures
will just focus on a tiny part corresponding to local $k$-algebras. 
 In particular we will not touch on arithmetic or
geometric aspects of the theory, such as the Galois theory of
J. Rognes, Brauer groups, or derived algebraic geometry. 

\subsection{Conventions}
We'll generally denote rings and ring spectra by letters like, $Q, R$ and $S$. 
Ring homomorphisms will generally go in reverse alphabetical order, as in 
$S\lra R \lra Q$. Modules will be denoted by letters like $L,M,N, \ldots$.

The ring of integers (initial amongst conventional rings)  is denoted $\Z$. 
The sphere spectrum (initial amongst ring spectra)  is denoted $\bbS$. 

After the inital sections we will use the same notation for a conventional ring
and the associated Eilenberg-MacLane spectrum. Similarly for modules. 

Generally, $M\tensor_{R} N$ will denote the left derived tensor product, and 
$\Hom_R(M,N)$ the module of derived homomorphisms. Similarly, fibre sequences, 
cofibre sequences, pullbacks and pushouts will be derived.

\subsection{Grading conventions.}
We will have cause to discuss homological and cohomological gradings.
Our experience is that this is a frequent source of confusion, so we adopt
the following conventions. First, we refer to lower gradings as {\em degrees}
and upper gradings as {\em codegrees}. As usual, one may convert gradings to
cogradings via the rule $M_n=M^{-n}$.
Thus both chain complexes and cochain
complexes have differentials of degree $-1$ (which is to say, of codegree $+1$).
This much is standard.  However, since we need to deal with both homology
and cohomology it is essential to have separate notation for
homological suspensions ($\Sigma^i$) and  cohomological suspensions
($\Sigma_i$): these are defined by
$$(\Sigma^iM)_n=M_{n-i} \mbox{ and } (\Sigma_iM)^n=M^{n-i}.$$
Thus, for example, with reduced chains and cochains  of a based
space $X$,  we have
$$C_*(\Sigma^iX)=\Sigma^iC_*(X) \mbox{ and } C^*(\Sigma^iX)=\Sigma_iC^*(X).$$

\section{Ring and module spectra I (motivation)}
\label{sec:ringspectra1}
Most of the generalities can take place in any category with suitable formal structure: we need a symmetric monoidal 
product so that we can discuss rings and modules over them, and we
need to be able to form tensor products over a commutative ring, 
and we will want a  homotopy theory which is stable so that we can form a triangulated homotopy 
category. In fact we will work with a category of spectra in the sense of homotopy theory; the construction is 
sketched in Appendix \ref{sec:constructspectra}, and the present section gives a brief orientation. A more extensive introduction 
designed for commutative algebraists is given in \cite{spectra}. 

\subsection{Additive structure}
We may start from the observation  that the homotopy theory of highly connected spaces is simpler than that
of general spaces. By 
suspending a space we may steadily simplify the homotopy theory, but because cohomology theories 
have suspension isomorphisms, we do not lose any additive cohomological information: spectra capture 
the limit of this process.  Thus spectra
are  a sort of abelianization of spaces where behaviour has a more algebraic formal flavour. Associated 
to any based space $X$ there is a suspension spectrum, and arbitrary spectra can be built from those of 
this form. The other  important source of spectra is as the representing objects for cohomology theories. 
If $E^*(X)$ is a  contravariant functor of a based space $X$ which satisfies the Eilenberg-Steenrod axioms, then 
there is a spectrum $E$ so that the equation 
$$E^*(X)=[X,E]^*$$
holds. This has the usual benefits that one can then apply geometric constructions to cohomology theories and one can argue more easily by universal examples. 

\subsection{Multiplicative structure}
Having taken the step of representing cohomology theories by spectra,  one may ask if good formal behaviour of the functor $E^*(\cdot) $ is reflected in the 
representing spectrum $E$. For our purposes the most important piece of structure is that of being a commutative
ring, and we would like to say that a cohomology theory whose values on spaces are commutative rings
is represented by a spectrum which is a commutative ring in the category of spectra. This is true remarkably often. 

In order to do homotopy theory we need a  Quillen model structure on the category of spectra,  
and to have commutative rings in this setting
we need a symmetric monoidal smash product,  with unit the sphere
spectrum $\bbS$ (the suspension spectrum of the two point space
$S^0$),  so that the two structures are compatible in a way elucidated by 
Schwede and Shipley \cite{SchwedeShipley}. In retrospect
it seems strange that such models were not constructed until the
1990s, but several such models are  now known, and they give
equivalent theories. We sketch the construction of symmetric spectra
in Section  \ref{sec:constructspectra}. 

In this context, it makes sense to ask for a cohomology theory to be
represented by a commutative ring spectrum $R$ (i.e., $R$ comes with a
multiplication map $\mu: R \sm R \lra R$ and a unit map $\bbS \lra R$
 making $R$ into a commutative monoid). Many of 
the important examples do have this structure. The most
obvious example from this point of view is the sphere spectrum $\bbS$
itself. 
This is the initial ring in the category of spectra, and the smash product plays the role of tensor product over $\bbS$.  We will describe a number of examples below, but for the present we continue with the general formalism just assuming 
that $R$ is a commutative ring spectrum. 

\subsection{Modules}
We may consider module spectra over $R$, and there is a model category
structure on $R$-modules; furthermore, since $R$ is commutative, there is a tensor product of $R$-modules
formed in the usual way from the tensor product over the initial ring $\bbS$ (i.e., as the coequalizer of the two maps
$M\tensor_{\bbS} R\tensor_{\bbS} N \lra M\tensor_{\bbS} N$). From the good formal properties of the original category, this category 
of $R$-modules is again a model category with a compatible symmetric monoidal product. This  has an associated 
homotopy category $\Ho(R-mod)$ and will be the context in which we work. 

\subsection{Reverse approach}
Commutative algebraists may approach spectra from the algebraic direction. Traditional commutative algebra considers commutative rings $A$ and modules over them, but some constructions make it natural to extend
further  to considering chain complexes of $A$-modules; the need to 
consider robust, homotopy invariant properties leads to the derived 
category $\sfD(A)$. Once we admit chain complexes, it is natural to 
consider the corresponding multiplicative objects, differential graded algebras (DGAs).
Although it may appear inevitable, the real justification for this 
process of generalization is the array  of naturally occurring examples. 

The use of spectra is a natural extension of this process. Shipley has
shown \cite{Shipley} that associated to any 
 DGA $A$ there is a ring spectrum $HA$, so that the derived 
 category $\sfD (A)$ is equivalent to the homotopy category of $HA$-module
 spectra. Accordingly we can view ring spectra
as generalizations of DGAs and categories of module spectra as 
flexible generalizations of the derived category.
Ring spectra extend the notion of rings, module spectra extend the notion
of (chain complexes of) modules over a ring, and the homotopy category of module spectra extends the
 derived category. 
Many ring theoretic constructions extend to ring spectra, 
and thus extend the power of commutative algebra to a vast new supply of naturally
occurring examples. Even for traditional rings, the new perspective is
often enlightening, and thinking in terms of spectra makes
 a number of new tools available. Once again the only compelling justification 
for this inexorable process of generalization  is the array of interesting,  naturally occurring 
examples, some of which will be described later in these lectures. 

\section{Ring and module spectra II (construction)}
\label{sec:constructspectra}

The purpose of this section  is to outline the construction of symmetric  spectra. 
The details will not be needed for the lectures, and  readers familiar
with spectra can comfortably omit it. The point is to reassure readers
that spectra are rather concrete objects. 

\subsection{Naive spectra}
\label{subsec:naive}
This subsection is designed to explain the idea behind spectra: where
they came from and why they were invented. Those already familiar
spectra should skip directly to Subsection \ref{subsec:symmetric},
which describes symmetric spectra.

The underlying idea is that spectra are just stabilised spaces and
the bonus is that they represent cohomology theories. This definition 
is perfectly good for additive issues, but it does not have a symmetric monoidal smash product, 
so is not adequate for commutative algebra. 

\begin{defn}
\label{defn:spectra}
A {\it spectrum\/} $E$ is a sequence of based spaces $E_k$ for  $k\ge 0$ 
together with structure maps
$$
\sigma\colon \Sigma E_k \to E_{k+1}, 
$$
where $\Sigma E_k=S^1\sm E_k$ is the topological suspension. 
A map of spectra $f\colon E\to F$ is a sequence of maps so that the
squares
$$
\begin{array}{ccc}
\Sigma E_k  &\stackrel{\Sigma f_k}\longrightarrow &  \Sigma F_k \\
\downarrow && \downarrow \\
E_{k+1}    &\stackrel{f_{k+1}}\longrightarrow   &   F_{k+1}
\end{array}
$$
commute for all $k$.
\end{defn}


\begin{example}
If $X$ is a based space we may define the suspension spectrum 
$\Sigma^\infty X$ to have $k$th term $\Sigma^k X$ with the structure maps
being the identity.
\vskip .1in

\noindent
{\bf Remark:}
It is possible to make a definition of homotopy immediately, but this does not 
work very well for arbitrary spectra. Nonetheless it will turn out that for 
finite CW-complexes $K$, maps out of a suspension  spectrum can be
easily described. We reserve $[\cdot, \cdot]$ for homotopy classes of
maps between spectra (sometimes called `stable maps'), so we write
$[A,B]_{\un}$  for homotopy classes of (unstable) maps from a based space $A$ to
a based space $B$. With this notation, we have   
$$
[\Sigma^\infty K,E]=\colim_{ k} [\Sigma^kK,E_k]_{\un}.
$$
In particular
$$
\pi_n(E):=[\Sigma^\infty S^n,E]=\colim_{ k} [S^{n+k},E_k]_{\un}.
$$
For example if $E=\Sigma^\infty L$ for a based space $L$, 
we obtain a formula for  the {\em stable homotopy groups} of $L$
$$
\pi_n(\Sigma^\infty L)=\colim_{ k} [\Sigma^kS^n,\Sigma^kL]_{\un}, 
$$
 By the Freudenthal suspension theorem, this is the common 
stable value of the groups  $[\Sigma^kS^n,\Sigma^kL]_{\un}$ for large $k$.
Thus spectra have captured stable homotopy groups.
\end{example}

\begin{construction}
We can suspend spectra by any integer $r$, defining $\Sigma^rE$ by
$$
(\Sigma^rE)_k= \begin{cases}
E_{k-r} \quad & k-r\ge0 \\
pt \quad & k-r <0.
\end{cases}
$$

Notice that if we ignore the first few terms,  $\Sigma^r$ is inverse to $\Sigma^{-r}$.
Homotopy groups involve a direct limit and therefore do not see these first few
terms. Accordingly,   once we invert homotopy isomorphisms,  the suspension functor becomes
an equivalence of categories. Because suspension is an equivalence,  we say that 
 we have a {\em stable} category.

\end{construction}

\vskip .1in

\begin{example} {\em (Sphere spectra)} We write
$\bbS =\Sigma^{\infty}S^0$ for the 0-sphere because of its special role, and
then for an arbitrary integer $r$ we define
$$
S^r=\Sigma^r \bbS  .
$$
Note that $S^r$ now has meaning for a space and a spectrum for $r \geq 0$, 
but since we have an isomorphism $S^r\cong \Sigma^\infty S^r$ of spectra
for $r\ge0$ the ambiguity is not important.
We extend this ambiguity, by often suppressing $\Sigma^\infty$.
\end{example}

\begin{example}
{\it Eilenberg-MacLane spectra\/}.
An Eilenberg-MacLane  space of type $(M,k)$ for an abelian group $M$ and $k \geq 0$
is  a CW-complex $K(M,k)$ with $\pi_k(K(M,k))=M$ and $\pi_n(K(M,k))=0$ for 
$n\neq k$; any two such spaces are homotopy equivalent.
It is well known that in each degree ordinary cohomology is
 represented by an Eilenberg-MacLane space. Indeed, for any CW-complex $X$,  we have 
$ H^k(X;)M=[X,K(M,k)]_{\un}$. In fact,   this sequence of spaces, as $k$ varies,  
assembles to make a spectrum.

To describe this, first note that the suspension functor $\Sigma$  is defined
 by smashing with the circle $S^1$, so it is left adjoint to the loop 
functor $\Omega$ defined by $\Omega X := \map (S^1,X)$ (based loops
with the compact-open topology). In fact there is a homeomorphism
$$
\map(\Sigma W,X)=\map (W \sm S^1, X) \cong \map (W , \map (S^1,X))=\map (W , \Omega X)
$$
 This passes to homotopy, so  looping shifts homotopy
in the sense that $\pi_n(\Omega X)=\pi_{n+1}(X)$. We conclude that there
is a homotopy equivalence
$$
\tilde{\sigma}: K(M,k)\overset{\simeq}{\to} \Omega K(M,k+1) , 
$$
and hence we may obtain a spectrum 
$$
HM=\{K(M,k)\}_{k \geq 0} 
$$
where the bonding map 
$$
\sigma\colon\Sigma K(M,k)\to K(M,k+1) 
$$
is adjoint to $\tilde\sigma$. Thus we find
$$
[\Sigma^r \Sigma^\infty  X, HM]= 
\colim_k [\Sigma^r\Sigma^kX, K(M,k)]_{\un}=
\colim_k  H^k(\Sigma^r \Sigma^kX;M)= H^{-r}(X;M).
$$
In particular the Eilenberg-MacLane spectrum has homotopy in a single degree like the 
spaces from which it was built: 
$$\pi_k(HM)= \begin{cases}
M \quad k=0 \\
0 \quad k\ne 0. \end{cases}$$
\end{example}

\subsection{Symmetric spectra \cite{HSS, Schwedesymmetric}}
\label{subsec:symmetric}

Symmetric spectra give an elementary and combinatorial construction of
a symmetric monoidal category of spectra. This is excellent for an
immediate access to the formal properties, but to be able to calculate
with symmetric spectra and to relate them to the rest of homotopy
theory one needs to understand  the construction of the {\em homotopy}
category. This is somewhat indirect, and Subsection \ref{subsec:naive}
was intended as a motivational substitute. 
 
It is usual to give a fully combinatorial construction of  symmetric spectra, by 
basing them  on simplicial sets rather than on  topological spaces.

\begin{defn}
(a)  A {\it symmetric sequence\/} is a sequence
$$
E_0,E_1,E_2,\dots,
$$
of pointed simplicial sets with basepoint
 preserving action of the symmetric group $\Sigma_n$ on $E_n$.

\vskip .1in

\noindent
(b) We may define a tensor product $E \tensor F$ of symmetric sequences $E$ and
$F$ by  
$$
(E\otimes F)_n:= \bigvee_{p+q=n} (\Sigma_{n})_+ \wedge_{\Sigma_p\times \Sigma_q} (X_p\wedge Y_q), 
$$
where the subscript $+$ indicates the addition of a disjoint basepoint.
\end{defn}

It is elementary to check that this has the required formal behaviour.

\begin{lemma}
The product $\otimes$ is symmetric monoidal with unit
$$
(S^0,*,*,*,\dots).
$$
\end{lemma}

\begin{example} The sphere is the symmetric sequence 
$\bbS:=(S^0,S^1,S^2,\dots)$. Here  $S^1=\Delta^1/ \Delta^1 $ is the simplicial circle 
and the higher simplicial spheres are defined by taking smash powers, 
so that  $S^n=(S^1)^{\wedge n};$ this also explains the actions of the symmetric groups.

It is elementary to check that {\em the sphere  is a commutative monoid}  in the category of 
 symmetric sequences.
\end{example}

\begin{defn}
A {\it symmetric spectrum $E$\/} is a left $\bbS$-module in symmetric sequences.

Unwrapping the definition, we see that this means $E$ is given by 

(1) a sequence $E_0,E_1,E_2,\dots$ of simplicial sets, 

(2) maps $\sigma\colon S^1\wedge E_n\to E_{n+1}$, and 

(3) basepoint preserving  left actions of $\Sigma_n$ on $E_n$ which are 
compatible with the actions in the sense that the composite maps 
$S^p\wedge E_n \to E_{n+p}$ are $\Sigma_p \times \Sigma_n$ equivariant.
\end{defn}

\begin{defn}
The smash product of symmetric spectra
is
$$
\diagram
E \sm_{\bbS} F:=\mathrm{coeq} (E \otimes \bbS \otimes F \rto<0.4ex> \rto<-0.4ex>& E \otimes F).
\enddiagram  $$
\end{defn}

\begin{prop}
The tensor product over $\bbS$ is a symmetric monoidal product on the category of symmetric
spectra.
\end{prop}

It is now easy to give the one example most important to commutative 
algebraists.
\begin{example}
For any abelian group $M$, we define the Eilenberg-MacLane symmetric spectra.
For a set  $T$ we write $M \otimes T$ for the $T$-indexed sum of copies
of $M$; this is natural for maps of sets and therefore extends to an 
operation on simplicial sets. We may then define the Eilenberg-MacLane
symmetric spectrum
$HM:=(M\tilde{\otimes} S^0,M\tilde{\otimes} S^1, M\tilde{\otimes}
S^2,\dots)$, where $\tilde{\tensor}$ means the tensor product of an
abelian group and a simplicial set with the basepoint generator set to
zero.  It is elementary to check that if $R$ is a commutative ring, then  $HR$ is a monoid 
in the category of $\bbS$-modules, and if $M$ is an $R$-module, $HM$ is a module over
$HR$.
\end{example}

Next one may  define  the stable model structure on symmetric spectra;
this is done in several steps (Section 3 of \cite{HSS}), but we may
summarise it by saying that  the weak equivalences (``stable equivalences'') are maps $f: X\lra Y$
which induce isomorphisms in all cohomology theories. The homotopy
category of symmetric spectra is obtained by inverting stable
equivalences.

\section{The three examples}

The point of the present section is to introduce our three basic
classes of examples. Classical commutative local Noetherian rings,
rational cochains on simply connected spaces and mod $p$ cochains on 
classifying spaces of finite groups.

The importance of the category of symmetric  spectra \cite{HSS} is that it admits
a symmetric monoidal smash product compatible with the model
structrures. Given this, we can start to do algebra
with spectra:  choose  a ring spectrum $\boldR$ (i.e., a monoid in the category of spectra), 
form the category of $\boldR$-modules or $\boldR$-algebras and 
then pass to homotopy. We may then attempt to use algebraic methods and
intuitions to study $\boldR$ and its modules. 

\subsection{Classical commutative algebra}
We explain why the classical derived category is covered by the context of spectra. 

In Section \ref{sec:constructspectra} we described a functorial construction for symmetric spectra taking an abelian group  $M$ and
giving the Eilenberg-MacLane spectrum $HM$, which is characterised up to homotopy by the property 
that $\pi_*(HM)=\pi_0(HM)=M$. For symmetric spectra the construction
is lax symmetric monoidal, so that
if $A$ is a commutative ring, $HA$ is a commutative ring spectrum. Furthermore the construction gives 
a  functor $A\text{-modules } \longrightarrow HA\text{-modules}$. 
Passage to homotopy groups gives a functor $\Ho($HA$-mod) \to \Ho(A$-modules$)=\sfD(A)$ and 
in fact the model categories are equivalent.

\vskip .1in

\begin{thm} (Shipley \cite[1.1, 2.15]{Shipley}) There is a Quillen equivalence
between the category of differential graded $\Z$-algebras and the category of $H\Z$-algebras 
in spectra. 

If we choose a DGA $A$ and the corresponding $H\Z$-algebra $HA$, 
there is a Quillen equivalence between  the category of $A$-modules and the category of $HA$-modules,
and hence in particular a triangulated equivalence
$$
\sfD(A)=\Ho(A\text{-modules}) \simeq \Ho(HA\text{-modules})=\sfD(HA)
$$
of derived categories. 
\end{thm}

The commutative case is complicated by the fact that there are no
model structures on commutative algebras in general. However, if $A$ is commutative, $HA$ may be
taken to be commutative and the Quillen equivalence between
$A$-modules and $HA$-modules, tensor products over $A$ and over $HA$
correspond. 

We sometimes use Shipley's result to excuse the omission of  the letter $H$ indicating spectra. 
In this translation homology in the classical context of chain complexes corresponds to homotopy in 
the context of spectra: $H_*(M)=\pi_*(HM)$.

Now that we can view classical commutative rings as  commutative ring spectra, we can attempt to 
extend classical notions to the context of spectra. From one point of view, we should first attempt to understand
the analogues of local rings before attempting to look at more geometrically complicated ones. Accordingly, 
in most of the lectures we will assume the commutative ring $A$ is  local, with residue field $k$.

\subsection{Cochains on a space}
\label{eg:cochains}
In the category of spectra, we may solve the commutative cochain problem. 
More precisely,  for any space $X$ and a commutative ring $k$,  we may form the function spectrum 
$C^*(X;k)=\text{map}(\Sigma^\infty X_+, Hk)$. It is obviously  an
$Hk$-module, and we may combine the commutative multiplication on $Hk$
with the  cocommutative diagonal on $X_+$ to see that $C^*(X;k)$ is a commutative $Hk$-algebra.
The notation is chosen because it is a model for the cochains in the sense  
$$
\pi_*(\text{map}(\Sigma^{\infty}X,Hk))=H^*(X;k).  
$$
The commutative algebra of $C^*(X;k)$ is one of the main topics for these lectures, and we
will omit the coefficient ring $k$ when it is clear from the context. 

We then use algebraic
behaviour of this commutative ring to pick out interesting classes of
spaces. In accordance with the principle that $C^*(X;k)$ is a sort of ring
of functions on $X$, we simplify terminology and say that $X$ has a property
P over $k$ if the commutative ring spectrum $C^*(X;k)$ has the
property P.

We will focus particularly on two special cases. 

\begin{example} {\em (Rational case.) }
We will refer to the special case $k=\Q$ with $X$ simply connected as
the rational case. Indeed this may be treated by classical means with
the Thom-Sullivan construction of piecewise polynomial forms giving a commutative DGA model for
$C^*(X;\Q)$. There is an enormous literature studying rational
homotopy theory by homotopy invariant commutative algebra methods. 
The book \cite{FHT} of F\'elix, Halperin and Thomas is an excellent
starting point. The `Looking glass' paper \cite{AvramovHalperinLookingGlass}
expounds the philosophy that algebra and topology are imperfect
reflections of each other and gives numerous profound examples of
it. Because we  permit more general ring spectra, we are more restricted in
theorems but have a larger range of examples. 
\end{example}

\begin{example}{\em (Modular representation theory.)}
We will also focus particularly on the case $k=\FF_p$, when $X=BG$
for a finite group $G$. We recall that the classifying space $BG$ of
principal $G$-bundles  is characterised (when $G$ is finite) 
by the fact that it has fundamental group $G$ and all other homotopy
groups trivial. The interest in this special case comes from the fact
that its coefficient ring 
$$\pi_*(C^*(BG;\FF_p))=H^*(BG;\FF_p)=\Ext_{\FF_p G}^*(\FF_p, \FF_p)$$
is the group cohomology ring.

There is an operation of $p$-completion due to Bousfield-Kan \cite{BK}
which behaves well for a large class of spaces (the {\em $p$-good}
spaces). In particular, all connected spaces with finite fundamental
group are $p$-good, which will cover our examples. For $p$-good spaces
$C^*(X;\FF_p)=C^*(X_p^{\wedge};\FF_p)$, and 
accordingly we will generally assume that $X$ is $p$-complete when $k=\FF_p$.
The space $X=BG$ is already $p$-complete if $G$ is a $p$-group, but in
general $BG_p^{\wedge}$ has infinitely many homotopy groups and its
funademental group is the largest $p$-quotient of $G$.
\end{example}

\section{Finiteness conditions}
\label{sec:context}

The point of this section is to introduce a variety of  finiteness conditions we
may impose. The essential 
limitation of all we do is that it is based on {\em local} rings in commutative algebra. We will not
discuss the new and interesting  features that can arise when there are many maximal ideals.

\begin{context}
The main input is a map $R\lra k$ of ring spectra with notation suggested 
by the case when $R$ is commutative local ring with residue field $k$. 
 \end{context}

For the most part, we  work in the homotopy category $\Ho(\Rmod)$ of left $R$-modules. 

\subsection{New modules from old.}
Two construction principles will be important to us.  The terminology
comes from Dror-Farjoun \cite{DrorFarjoun}, but in our stable context the behaviour is
rather simpler. 

If $M$ is an $R$-module we say that $X$ is {\em built} from $M$ (and write $M\builds X$)
 if $X$ can be  formed from $M$ by completing triangles, taking 
coproducts and retracts (i.e., $X$ is in the {\em localizing subcategory} 
generated by $M$). We refer to objects built from $M$ as {\em $M$-cellular}, 
and write  $\cell (R,M)$ for the resulting full subcategory of $\Ho (\Rmod )$.
An {\em $M$-cellular approximation} of $X$ is a map $\cell_M(X) \lra X$ where
$\cell_M(X)$ is $M$-cellular and the map is an $\HomR (M,\cdot)$-equivalence.

 We say that $X$ is {\em finitely built} from $M$ (and write $M\finbuilds X$)
if only finitely many steps  and finite coproducts are necessary (i.e., $X$ is in  the {\em thick} 
subcategory generated by $M$).

Finally, we say that $X$ is {\em cobuilt} from $M$ if $X$ can be 
formed from $M$ by completing triangles, taking 
products and retracts (i.e., $X$ is in the {\em colocalizing subcategory} 
generated by $M$).

\subsection{Finiteness conditions.}

We say $M$ is {\em small} if the natural map 
$$\bigoplus_{\alpha}[M, N_{\alpha}]\lra [M,\bigoplus_{\alpha} N_{\alpha}]$$
is an isomorphism for any set of $R$-modules $N_{\alpha}$. Smallness is 
equivalent to being finitely built from $R$. It is easy to see that 
any module finitely built from $R$ is small. For the reverse implication we use 
the fact that we can build an $R$-cellular approximation $\cell (R,M) \lra M$; 
this is an equivalence, and by smallness, $M$ is a retract of a finitely built 
subobject of $\cell (R,M)$. 

We sometimes require that  $k$ itself is small, but this is an extremely strong condition on $R$
and it is important to develop the
 theory under a much weaker condition. 

\begin{defn} \cite{DGI} 
\label{defn:proxysmall}
We say that $k$ is {\em proxy-small} if there is an object 
$K$ with the following properties
\bi
\item $K$ is small ($R\finbuilds K$)
\item  $K$ is finitely built from $k$ ($k\finbuilds K$) and
\item  $k$ is built from $K$ ($K\builds k$). 
\ei
\end{defn}

\begin{remark}
Note that the second and third condition imply that the $R$-module $K$ 
generates the same category as $k$ using
triangles and coproducts: $\cell (R,K)=\cell (R,k) . $
\end{remark}

One of the main messages of \cite{DGI} is that we may use the
condition that $k$ is proxy-small as a substitute for the Noetherian
condition in the conventional setting. This rather weak condition 
allows one to develop a very useful theory applicable in a large range
of examples. 

We can illustrate this by looking at the proxy-small condition in the classical case. 

\begin{example}   {\em (Conventional commutative algebra)}
Take $R$ to be a commutative Noetherian local ring in degree 0, with maximal 
ideal $\fm$ and residue field $k$. 

By the Auslander-Buchsbaum-Serre theorem, $k$ is  small if and only if $R$ is a regular 
local ring,  confirming that  the smallness of $k$ is a very 
strong condition. On the other hand, $k$ is  always proxy-small:  
we may take $K=K (\AL)$ to be the Koszul complex for a generating 
sequence $\AL$ for $\fm$ (see Appendix B). 

It is shown in \cite{tec} that $\cell (R,k)$ consists of objects whose 
homology is $\fm$-power torsion. 
\end{example}

\part{Morita equivalences}
Section \ref{sec:Morita} introduces the basic apparatus of the Morita
equivalences that concern us, along with their relation with torsion
functors and completion. Section \ref{sec:proxysmall} discusses the
rather weak condition that $k$ is proxy-small in our examples, and
shows it provides an inclusive framework. Section
\ref{sec:exterior} illustrates the power of the Morita framework by 
discussing the classification of ring spectra and DGAs with
coefficient rings exterior on one generator. 

\section{Morita equivalences}
\label{sec:Morita}

Morita theory studies objects $X$ of a
category $\C$ by considering maps from a test object $k$. More
precisely, $X$ is studied by considering $\Hom (k,X)$
as  a module over the endomorphism ring $\End (k)$. 
In favourable circumstances this may give rather complete information. 
 This is an instance of the philosophy that one gains insights
by looking at rings of operations.

In the classical situation,  $\C$ is an abelian category with infinite sums and $k$ 
is a small projective generator, and we find $\C$ is equivalent
to the category of $\End(k)$-modules \cite[II Thm 1.3]{BassK}.  
We will work with a stable model category rather than in an 
abelian category, and  $k$ will not necessarily be either small or
a generator. The fact that the objects of the categories are
spectra is unimportant except for the formal context it provides.
See \cite{SchwedeM} for a more extended account from the present point of view.

In fact,  two separate Morita equivalences play a role: two 
separate categories of modules over a commutative ring are both 
shown to be equivalent to a category of modules over the same non-commutative
ring.

This section is based on \cite{tec}, with augmentations from \cite{DGI}.

\subsection{First variant.}
To start with we introduce the  ring spectrum $\cE=\HombR
(\boldk,\boldk)$ of (derived) endomorphisms of $k$. This is usually
not a commutative ring. Morita theory  considers the relationship between the 
categories of left $\boldR$-modules 
and  of right $\cE$-modules. We have  the adjoint pair 
$$\RLadjunction{E}{\bRmod}{\modcE}{T}$$
defined by 
$$T(X):=X \tensorcE \boldk \mbox{ and } E(M):=\HombR (\boldk,M).$$ 
It is clear from the definition that $E(M)$ is an $\cE$-module, since
$\cE =\Hom_R(k,k)$ acts on the right through the factor $\boldk$ in
$EM=\HombR (\boldk, M)$. By the same token, it is clear that the ring 
$$\hat{R}=\Hom_{\cE}(\boldk, \boldk)$$
acts on the left of $TX=X\tensor_{\cE}\boldk$ through its action on
$\boldk$.
To see that this gives an $R$-module we note that there is a ring map 
$$R\lra \Hom_{\cE}(\boldk, \boldk)=\hat{R}$$
by `right multiplication' of $R$ on $\boldk$; thuse $TX$ is an
$R$-module by restriction.

\begin{thm}
 If $k$ is small, this adjunction gives equivalence 
$$\cell (R,k) \simeq \modcE$$
between the derived 
category of $\boldR$-modules built from $\boldk$
and the derived category of $\cE$-modules. 
\end{thm}

\begin{proof}
To see the 
unit $X \lra ETX=\HomR (k, X \tensorE k)$ is an equivalence, 
we note it is obviously an equivalence  for $X=\cE$ and hence
for any $X$ built from $\cE$, by smallness of $k$. The argument
for the counit is similar.
\end{proof}

\begin{remark}
If $k$ is not small, the unit of the adjunction may not be an equivalence. For 
example if $R=\Lambda (\tau)$ is exterior on a generator $\tau$ of degree $1$ then 
$\cE \simeq k[x]$ is polynomial on a generator of 
degree $-2$. As an $R$-module,  $k$ is of infinite projective dimension and
hence it is not small. In this case all $R$-modules are $k$-cellular, so that 
$\cell (R,k) =\Rmod$. Furthermore, the only subcategories of $R$-modules 
closed under coproducts
and triangles are the trivial category and the whole category. 
On the other hand the category of torsion $\cE$-modules is a proper 
non-trivial subcategory closed under coproducts and triangles. 

Exchanging roles of the rings, so that $R=k[x]$ and $\cE \simeq \Lambda (\tau)$, 
we see $k$ is small as a $k[x]$-module 
and $\cell (k[x],k)$ consists of torsion modules. Thus we deduce
$$\mbox{tors-$k[x]$-mod} \simeq \mbox{mod-$\Lambda (\tau)$}. \qqed$$
\end{remark}

In fact  the counit 
$$TEM =\HomR (k,M) \tensorE k \lra M$$ 
of the adjunction is of interest much more generally. Notice that 
any $\cE$-module (such as $\HomR (k,M)$) is built from 
$\cE$, so the domain is $k$-cellular. We say $M$ is {\em effectively 
constructible} from $k$ if the counit is an equivalence, because
$TEM$ gives a concrete and functorial model for the cellular approximation
to $M$.  Under the much weaker assumption of proxy smallness 
we obtain a very useful conclusion linking Morita theory to commutative
algebra. 

\begin{lemma}
\label{proxycell}  
Provided $\boldk$ is proxy-small, 
 the counit 
$$TEM =\HomR (k,M) \tensorE k \lra M$$ 
is $k$-cellular approximation, and hence
in particular any $k$-cellular object is effectively constructible from $k$. 
\end{lemma}

\begin{proof} 
We observed above that the domain 
is $k$-cellular. To see the counit is a $\HomR (k,\cdot )$-equivalence, 
consider the evaluation map 
$$\gamma: \HomR (k,X) \tensorE \HomR (Y,k) \lra \HomR (Y,X).$$ 
This is an equivalence if $Y=k$, and hence by proxy-smallness it is an 
equivalence if $Y=K$. This shows that the top horizontal in the  diagram
$$\begin{array}{ccc}
\HomR (k,X) \tensorE \HomR (K,k)& \stackrel{\simeq} \lra &\HomR (K,X)\\
\simeq \downarrow &&\downarrow = \\
\HomR (K,\HomR (k,X)\tensorE k) &\lra &\HomR (K,X)
\end{array}$$ 
is an equivalence. The left hand-vertical is an equivalence since $K$ is
small. Thus the lower horizontal is an equivalence, 
which is to say that the counit 
$$TEX=\HomR (k,X)\tensorE k \lra X $$ 
is a $K$-equivalence. By proxy-smallness,  this counit map is a $k$-equivalence.
\end{proof}

\begin{examples}
\label{egcell}
(i) If $R$ is a commutative Noetherian local ring, we recall in Appendix B
that the $k$-cellular approximation 
of a module $M$ is $\Gamma_{\fm}M =\Kinfty (\fm ) \tensorR M$, where $\Kinfty (\fm)$ is the
stable Koszul complex,  so we have
$$TEM \simeq \Kinfty (\fm) \tensor_R M.$$

(ii) If $R=C^*(X;k)$ it is not easy to say what the $k$-cellular approximation  is
 in general, but any bounded below module  $M$  is cellular.
\end{examples}

\subsection{Second variant.}

There is a second adjunction  between the derived 
categories of left $\boldR$-modules and  of right $\cE$-modules. 
In the first variant, $k$ played a central role as a left $R$-module
and a left $\cE$-module. In this second variant 
$$\ksharp :=\HombR (k,R)$$
plays a corresponding role: it is a right $R$-module and a right 
$\cE$-module. We have  the adjoint pair 
$$\adjunction{\Ep}{\bRmod}{\modcE}{C}$$
defined by 
$$\Ep (M):=\ksharp \tensorR M \mbox{ and } C(X):=\HomE (\ksharp ,X).$$ 

\begin{remark}
\label{defn:comp}
 If $k$ is small then 
$$\Ep (M) = \HomR (k,R) \tensorR M \simeq \HomR(k, M)=EM, $$
so the two Morita equivalences consider the left and right adjoints of the 
same functor.
\end{remark}

The unit of the adjunction $M \lra C\Ep (M)$ is not
very well behaved, and the functor $C\Ep$ is not even idempotent in general. 


\subsection{Complete modules and torsion modules} \label{sec:Rmodfunctors}

Even when we are not interested in the intermediate category 
of $\cE$-modules, several of the composite functors give interesting
endofunctors of the  category 
of $R$-modules.

\begin{lemma}
\label{cellsmashing}
If $k$ is proxy-small then $k$-cellular approximation  is smashing:
$$\cell_kM \simeq (\cell_k R)\tensorR M, $$
so that 
$$TEM=TE'M.$$
\end{lemma}

\begin{proof}
We consider the map 
$$\lambda : (\cell_kR )\tensorR M\lra R\tensorR M=M $$
and show it has the properties of cellularization. 

To start with, since $R$ builds $M$, it follows that 
$(\cell_kR) \tensorR R$ builds $(\cell_k R) \tensorR M$, and hence
$(\cell_kR) \tensorR M$ is $k$-cellular.

Since $k$ is proxy small, the  map $\lambda $ is $k$-equivalence if and only if
it is a $K$-equivalence. We may consider the class of objects $M$ for
which $\lambda$ is a $K$-equivalence. Since $K$ is small, it is closed
under arbitrary coproducts as well as triangles. Since the class obviously
contains $R$ itself, it contains all modules $M$.
\end{proof}

We therefore see by \ref{proxycell} and \ref{cellsmashing} that if $k$ is proxy-small 
$$\cell_k (M)=TE M =T\Ep M.$$
This is the composite of two left adjoints, focusing 
attention on its right adjoint $CE M$, and we note that
$$CE(M)=\HomR (\ksharp, \HomR (k, M))=\HomR (TE R, M).$$
By analogy with  Subsection \ref{subsec:loccomp} of Appendix \ref{sectoploccoh}, we may make the 
following definition.

\begin{defn} 
The {\em completion } of an $R$-module $M$ is the map
$$M \lra \HomR (TER, M)=CEM.$$
We say that $M$ is {\em complete} if the completion 
map is an equivalence. 
\end{defn}

\begin{remark}
By \ref{cellsmashing} we see that completion is idempotent. 
\end{remark}

We  adopt the notation 
$$\Gk M := TE'M$$
and
$$\Lk M := C EM. $$
This is  by analogy with the case of commutative algebra through 
the approach of Appendix C, where
$\Gk =\Gamma_I$ is the total 
right derived functor of the $I$-power torsion functor and 
$\Lk =\Lambda^I$ is the total left derived functor of the completion 
functor (see \cite{AJL1,AJL2} for the  context of commutative rings).

It follows from the adjunctions described earlier in this section 
that $\Gk$ is left adjoint to $\Lk$ as endofunctors of the category of 
$R$-modules: 
$$\HomR (\Gk M, N)=\HomR (M, \Lk N)$$ 
for $R$-modules $M$ and $N$. Slightly more general is the following observation.

\begin{lemma}
If $k$ is proxy-small, $\Gk$ and $\Lk$ give an adjoint equivalence
$$\cell (R,k) \simeq \mathrm{comp}\!-\!\Rmod,$$
where $\mathrm{comp}\!-\!\Rmod$ is the triangulated subcategory 
of $R-mod$ consisting of complete modules.
\end{lemma}

\begin{proof}
We have
$$T\Ep M \simeq T E M \simeq \Gk M \simeq \Gamma_KM, $$
and 
$$CE M \simeq \HomR (\Gk R,M) \simeq \HomR (\Gamma_K R,M),$$
so it suffices to prove the result when $k$ is small. When 
$k$ is small the present adjunction is the composite of two  adjoint pairs of
equivalences. We have seen this for the first variant, and the second
 variant is proved similarly by arguing that the unit and counit
are equivalences. 
\end{proof}

\subsection{dc-completeness}
We have  explained the importance of passing from $R$ to
$\cE=\Hom_R(k,k)$, and we have also given an example where we then
took  $\cE$ as our input ring, meaning that we need to consider the
ring  $\hat{R}=\Hom_{\cE}(\boldk, \boldk)$. Furthermore, the double
centralizer ring homomorphism 
$$\kappa : R\lra \Hom_{\cE}(\boldk, \boldk)=\hat{R}$$
played a role at the very start, in defining the $R$-module structure
on $TX$. 

We say that $R$ is {\em dc-complete} if $\kappa$ is an equivalence.

\section{Proxy-smallness in the examples}
\label{sec:proxysmall}
We establish the appropriate finiteness and completeness properties
in our principal examples.

\subsection{Commutative local rings}
When $R$ is a conventional commutative local Noetherian ring, we have
seen that $k$ is always proxy-small, and that the Koszul complex for a
set of generators of $\fm$ provides the small proxy for $k$. 

Next, we note that  
$$\pi_*(\cE)=\pi_*(\Hom_R(k,k))=\Ext_R^*(k,k), $$ 
which is a ring whose importance is very familiar. We will see in Section \ref{sec:exterior} that as a DGA, the endormorphisms $\cE =\Hom_R(k,k)$ contain
 considerably more information. 

We also note that in this case, $R$ is dc-complete if and only if it
is $\fm$-complete in the usual sense. Indeed, we calculate
$$\begin{array}{rcl}
R_{\fm}^{\wedge}&=&\holim_s R/\fm^s\\
&=&\holim_s \Hom_{\cE}(\Hom_R(R/\fm^s, k), k)\\
&=&\Hom_{\cE}(\hocolim_s \Hom_R(R/\fm^s, k), k)\\
&=&\Hom_{\cE}(k,k)
\end{array}$$
Here the first equivalence uses the Mittag-Leffler condition on $\{
R/\fm^s\}_s$, and the last condition uses the fact that 
$$\hocolim_s \Hom_R(R/\fm^s, k)\simeq \Gamma_{\fm}k\simeq k. $$

More generally, if $R$ is a commutative Noetherian
ring and $I$ is an ideal, we may consider $R\lra R/I$ (i.e., replacing
the field $k$ by a more complicated ring). It turns out that the
double centralizer map is $I$-adic completion whenever $R/I$ is
proxy-small.

\subsection{Cochains on a space}
\label{subsec:cochains}
In the setting $R=C^*(X; k)$ of cochains on a space, the Eilenberg-Moore spectral theorem shows that
$$\cE=\Hom_{C^*(X;k)}(k,k)\simeq C_*(\Omega X; k), H_*(\cE)=H_*(\Omega X; k)$$
provided that either (Case 0) $X$ is simply connected or (Case $p$) $X$ is connected,  $\pi_1(X)$ is a finite 
$p$-group $k=\Fp$  and $X$ is $p$-complete \cite{DwyerEM}. 
This immediately gives a very rich source of examples that we will revisit frequently. 

Provided we have either Case $0$ or Case
$p$, then $C^*(X)$ is dc-complete. Indeed, the Rothenberg-Steenrod equivalence 
$$\Hom_{C_*(\Omega X)}(k, k)\simeq C^*(X)$$
holds for any connected space $X$. Accordingly, provided the Eilenberg-Moore
equivalence holds,  the double centralizer map $\kappa$ is an equivalence. 

Finally we may consider proxy-smallness.
\begin{lemma}
\label{lem:proxysmall}
(i) If $\dim_k H_*(\Omega Y)$ is finite  then $k$ is small over
$C^*(Y)$.

\noindent 
(ii) If $Y$ falls under  Case 0 or Case $p$, there is a map $X\lra Y$ with homotopy fibre $F$, and 
$\dim_k H_*(\Omega Y)$ and $\dim_k H^*(F)$ are finite, then $k$ is
proxy-small over $C^*(X)$.
\end{lemma}
\begin{proof}
Since $\dim_kH_*(\Omega Y)$ is finite, 
$$k\finbuilds C_*(\Omega Y).$$
 Applying $\Hom_{C_*(\Omega Y)}(\cdot ,
k)$ we see 
$$C^*(Y)\simeq \Hom_{C_*(\Omega Y)}(k,k)\finbuilds \Hom_{C_*(\Omega
  Y)}(C_*(\Omega Y),k)\simeq k. $$
This proves Part (i). 

For Part (ii), let $R=C^*(X)$, $S=C^*(Y)$ and note
$C^*(X)\tensor_{C^*(Y)}k\simeq C^*(F)$ by the Eilenberg-Moore
theorem. Now take $K= C^*(F)$. We note that $k\finbuilds K$ by
hypothesis. Since $S\finbuilds k$ by Part (i), we may apply
$R\tensor_S (\cdot)$ and deduce 
$$R=R\tensor_S S \finbuilds R\tensor_S k=C^*(F).$$
Finally, $C^*(F)$ builds $k$ by killing homotopy groups, since $F$ is
connected.  
\end{proof}

\begin{cor}
\label{cor:proxysmallegs}
(Case 0) If $X$ is simply connected and $k=\Q$ and if $H^*(X)$ is
Noetherian then $\Q$ is proxy small over $C^*(X; \Q)$. 

\noindent
(Case $p$) If $k=\Fp$ and $X$ is the $p$-completion of $BG$  then $\Fp$ is proxy small over $C^*(X; \Fp)$. 
\end{cor}

\begin{proof}
For Case $0$, since $H^*(X)$ is Noetherian, we may choose a polynomial
subring $P=\Q [V]$ on finitely many even degree generators. We have
$\Q [V]=H^*(Y)$ where $Y=K(V)$ is the corresponding product of even
Eilenberg-MacLane spaces. The inclusion $P=K[V]\lra H^*(X)$ is realized by
a map $X \lra Y$. The space $\Omega Y$ is a product of finitely many
odd Eilenberg-MacLane spaces, and rationally this is a product of
spheres, so Part (i) of Lemma \ref{lem:proxysmall}
applies. Furthermore since $H^*(X)$ is a finitely generated $P$
module, it has a finite resolution by finitely generated projectives
and the spectral sequence for calculating
$\pi_*(C^*(X)\tensor_{C^*(Y)}\Q) =H^*(F)$ has finite dimensional $E_2$
term and so the hypotheses of Part (ii) of Lemma \ref{lem:proxysmall}
hold. 

For Case $p$ we note that the $p$-completion of  $Y=BU(n)$ satisfies
the conditions of Part (i) of Lemma \ref{lem:proxysmall}. Now, given $G$,  choose $n$ so that $G$
admits a faithful representation in $U(n)$ and apply Part (ii) of
Lemma \ref{lem:proxysmall} to the $p$-completion of the fibration $U(n)/G\lra
BG\lra BU(n)$.

\end{proof}

\section{Exterior algebras}
\label{sec:exterior}

There are many ways to use Morita theory, but there is a very elementary and striking one which 
illustrates that it is significant in even in the  very simple case
when $\pi_*(\cE)$ is an exterior algebra on one generator. 

\subsection{Exterior algebras over $\bbF_p$ on a generator of degree
  $-1$}
\label{subsec:specialexterior}

To start with, we note that is easy to see that if $\cE$ is an
$\Fp$-algebra with $\pi_*(\cE) =\Lambda_{\Fp}(\tau)$ exterior on one generator, then $\cE$
is formal. However if $\pi_*(\cE) =\Lambda_{\Fp}(\tau)$ 
 but $\cE$ itself is only known to be a $\Z$-algebra, the
situation is considerably more complicated.  

We may use Morita theory to give a classification \cite{DGIexterior}
in the special case that the exterior generator is in degree $-1$. We
will return to discuss the more general case in Subsection
\ref{subsec:generalexterior}. 

\begin{thm}
Differential graded algebras $\cE$ with $H_*(\cE)$ exterior over
$\bbF_p$ on a single generator of degree $-1$ (up to
quasi-isomoprhism) are in bijective correspondence with complete
discrete valuation rings with residue field $\bbF_p$ (up to
isomorphism of rings). 
\end{thm}

\begin{proof}
The idea is to associate to any such $\cE$ the endomorphism DGA $R=\Hom_{\cE}(\bbF_p, \bbF_p)$.
Evidently the spectral sequence for calculating $H_*(R)$ collapses at
$E_2$ with value $\bbF_p [\overline{x}]$ for an element $\overline{x}$
of total degree 0. This shows $R$ is a filtered ungraded ring with
this as associated graded ring. One argues that it is commutative and
complete with residue field $\bbF_p$. If one starts with $R$, its
maximal ideal is principal, generated by an 
element $x$ and we may form a complex $\uFp=(R\stackrel{x}\lra R)$. The DGA  $\cE =\Hom_R(\uFp, \uFp)$ has
homology  exterior on a generator of degree $-1$, and the double centralizer completion map
$$R\lra \Hom_{\cE}(\bbF_p, \bbF_p)$$
is an equivalence.  
\end{proof}

\subsection{General algebras with exterior homotopy}
\label{subsec:generalexterior}

Let us consider ring spectra  $\cE$ with $\pi_*(\cE)=\Lambda_{\Fp}(\tau)$
where $\tau$ is of degree $d$. We have ring maps $\bbS \lra \Z \lra
\Z/p$, and we have already observed that if we restrict attention to
$\Z/p$-algebras, there is a unique such $\cE$. If $d=-1$  we
classified $\Z$-algebras in Subsection \ref{subsec:specialexterior},
and found there were a lot of them. If $d=0$ the answers for
$\bbS$-algebras and $\Z$-algebras are the same by Shipley's Theorem,
and  there is a unique such algebra. 

For $d\geq 1$ there is an obstruction theory, and the answer is given
by Dugger and Shipley \cite{DuggerShipley}. Indeed, we are considering
square zero extensions of $\Z/p$ by the bimodule $\Z/p$ in degree $d$, 
and hence (with $A=\bbS$ or $\Z$),  the $\A$-algebras  $\cE$
are classified by the Hochschild cohomology group  $HH^{d+2}(\Z/p | \A; \Z/p)$, with orbits under
$Aut (\Z/p)$ giving isomorphic algebras $\cE$. In fact the
Hochschild cohomology rings are polynomial and divided power algebras
on a single degree 2 generator:
$$HH^*(\Z/p | \Z; \Z/p)=\Z/p [x_2^{\Z}] \mbox{ and } HH^*(\Z/p | \bbS;
\Z/p)=\Gamma_{\Z/p} [x_2^{\bbS}]. $$
We conclude that for any odd $d \geq 1$ there is single formal algebra
and  for any even $d \geq 2$ there are precisely two isomomorphism
classes of  $\Z$-algebras and  precisely two isomomorphism
classes  of  $\bbS$-algebras. In each case one of them is formal and
one is not. 

The really striking phenomenon is that restriction 
along $\bbS \lra \Z$, does not induce a bijection of isomorphism
types. Indeed, one may check that the generator $x_2^{\Z}$ restricts
non-trivially (so we may take $x_2^{\bbS}$ to be that restriction). It
follows that there is a bijection until we reach $(x_2^{\Z})^p$, which restricts to
zero in $\Gamma (x_2^{\bbS})$. In other words, the two $\Z$-algebras 
$\cE, \cE'$ with $d=2p-2$ become equivalent as $\bbS$-algebras: these 
two inequivalent DGAs are topologicially equivalent. It is hard to be
explicit when $d$ is not small, but if $p=2$ and $d=2$, the DGA 
$\Z [e \st de=2]/(e^3, 2e^2)$ is not formal as a $\Z$-algebra but it
is formal as an $\bbS$-algebra. Finally, the $\bbS$-algebra $\cE$ with
$d=2p-2$ which is  not realizable as a  $\Z$-algebra, is the truncation of the $d$th
Morava $K$-theory.

\subsection{Classification of  free rational $G$-spectra.}
In another direction, working over the rationals we note that for any compact Lie group $G$ the DGA
$C_*(G)$ has homology $H_*(G)$ exterior on odd degree generators. It might not be apparent that it 
is formal. However, we may consider $R=\Hom_{C_*(G)}(k,k)\simeq C^*(BG)$; since this is commutative and has cohomology 
which is polynomial on even degree generators,  it is formal and hence equivalent to $H^*(BG)$. It follows that $C_*(G)\simeq H_*(G)$. This is the starting point for classifying
free rational $G$-spectra \cite{gfreeq, gfreeq2}.

\part{Regular rings}

We have now set up the basic machinery and the rest of the lectures will investigate
homotopy invariant counterparts of classical definitons in commutative
algebra of local rings. The justification consists of the dual facts that the new definition reduces to the old in 
the classical setting and that the new definition covers and
illuminates examples in new contexts. Regular local rings are the most 
basic and best behaved objects, so it makes sense to start with
these. 

Section \ref{sec:regular} recalls the three styles of definition from
commutative algebra and explains how to give homotopy invariant
versions, and then illustrates them in our main examples. Section
\ref{sec:normalizable} explains how we can use this class of ring spectra to
give a homotopy invariant counterpart of the notion of a 
finitely generated module, at least for `normalizable' ring spectra. 

\section{Regular ring spectra}
\label{sec:regular}
In commutative algebra there are three styles for a definition of
a regular local ring: ideal theoretic, in terms of the growth of
the Ext algebra and a version for modules. We begin in Subsection 
\ref{subsec:regcommalg} by recalling the commutative algebra, 
where the conditions are equivalent. We  then turn to other contexts, 
where the conditions may differ, and consider which one is most appropriate.  

\subsection{Commutative algebra.}
\label{subsec:regcommalg}
The following definitions are very familiar; we have introduced a slightly more
elaborate terminology to smooth the transition to other contexts. 
The prefix s- signifies that the definition is {\em structural}, the prefix g- signifies
that the definition is in terms of the {\em growth} rate, and the prefix m- signifies
that the definition is in terms of the {\em module} category.

\begin{defn}
(i) A local Noetherian ring $R$ is {\em s-regular} if the maximal ideal
is generated by a regular sequence.

(ii) A local Noetherian ring $R$ is {\em g-regular} if $\Ext_R^*(k,k)$ is
finite dimensional.

(iii) A local Noetherian ring $R$ is {\em m-regular} if every finitely generated
module is small in the derived category $\sfD(R)$.
\end{defn}

These are equivalent by the Auslander-Buchsbaum-Serre theorem; one
might think of this as the first theorem of homotopy invariant
commutative algebra.

\subsection{Regularity for ring spectra}
The one definition that is easy to adapt is g-regularity, at least when $k$ is a field. 
This is a basic input to  the entire theory.
It is also convenient to have a name for when the coefficient ring is regular. 

\begin{defn}
(i) We say that $R$ is {\em c-regular} when the coefficient ring $\pi_*(R)$ is regular.

(ii) We say that 
$R$ is {\em g-regular} if $\pi_*(\Hom_R(k,k))$ is finite dimensional
over $k$. 
\end{defn}

It is obvious from the spectral sequence 
$$\Ext_{R_*}^{*,*}(k,k)\Rightarrow \pi_*(\Hom_R(k,k))$$
that c-regular implies g-regular. The converse is very far from being true, as we shall 
see shortly.

It is more subtle to consider s-regular ring spectra.

The s-regularity condition on rings states that $\fm$ is generated by a regular 
sequence $x_1, \ldots , x_r$, so  that there are  short exact sequences
$$R\stackrel{x_1}\lra R \lra R/(x_1)$$
$$R/(x_1)\stackrel{x_2}\lra R/(x_1) \lra R/(x_1,x_2)$$
and 
$$R/(x_1, \ldots , x_{r-1})\stackrel{x_r}\lra R/(x_1, \ldots , x_{r-1}) \lra R/\fm =k.$$
In other words,  we may start with $R$, and successively  factor out a
regular element until we get to $k$. 

If we now remember the degrees, a single instance is the (additive) exact sequence  of modules
$$\Sigma^{n}R\stackrel{x}\lra R \lra R/(x). $$
In other words, the ring $R/(x) $ is equivalent to the Koszul complex $K(x)=(R\stackrel{x}\lra R)$. 
If we think multiplicatively this gives a cofibre sequence \footnote{If $S\lra R$ is a map of rings and we
have a map $R\lra k$, the {\em cofibre} ring is $Q=R\tensor_Qk$. This corresponds to the fact
that in  algebraic geometry and topology this cofibre ring is often  the ring of functions on the geometric fibre.}
$$R\lra K(x)\lra K(x)\tensor_Rk = \Lambda (\tau)$$
of rings, where $\tau$ is a generator of degree $n+1$. The definition is now obtained by iterating 
this construction. 

\begin{defn}
We say that $R$ is {\em s-regular} if there are cofibre sequences of rings
$$R=R_0\lra R_1\lra \Lambda_1, R_1\lra R_2\lra \Lambda_2, \cdots ,R_{r-1}\lra R_r\lra \Lambda_r,$$ 
with $R_r\simeq k$ and $\pi_*(\Lambda_i) $ exterior over $k$ on one
generator.  
\end{defn}


\begin{lemma}
If $R$ is s-regular then it is also g-regular.
\end{lemma}

\begin{proof}
We need to work our way along the sequence of cofibrations. Since $R_r\simeq k$ is obviously g-regular, it suffices to oberve that 
given a cofibre sequence $C\lra  B \lra \Lambda$ with $B$   g-regular,
then $C$ is also g-regular. Suppose then that $\Hom_B(k,k)$ is
finitely built from $k$. Since $k\finbuilds \Lambda$ it follows that
$$k\finbuilds \Hom_B(k, k)\finbuilds \Hom_B(\Lambda, k)
\simeq \Hom_k(k\tensor_B\Lambda, k).$$ 
However 
$$
k\tensor_B \Lambda 
\simeq 
k\tensor_B B \tensor_C k 
\simeq 
k\tensor_C k .$$
so that 
$$\Hom_k(k\tensor_B\Lambda, k)\simeq 
\Hom_k(k\tensor_Ck,k)\simeq \Hom_C(k,k),$$
showing $\Hom_C(k,k)$ is finitely built by $k$ as required. 
\end{proof}

However, there is no good notion of  m-regularity for ring spectra.  The problem is that  we do not have a homotopy invariant definition of finite 
generation in general. However, we will turn this around, and  in
Section \ref{sec:normalizable} we will show that one can 
use the supposed equivalence of g-regularity and the putative m-regularity conditions to 
{\em define} finite generation.

\subsection{Regularity for rational spaces}

We have seen in Subsection \ref{subsec:cochains},  
that for simply connected rational spaces $X$, if 
$R=C^*(X;\Q)$ we have $\pi_*(\cE)=H_*(\Omega X)$. Accordingly a
g-regular rational space is one with $H_*(\Omega X)$ finite
dimensional. 

\begin{prop}
A simply connected rational space $X$ is g-regular if and only if it
is a finite product of even Eilenberg-MacLane spaces. Its  cohomology 
is therefore polynomial on even degree generators and it is also s-regular.
\end{prop}

\begin{proof} 
To start with,  we note that an even Eilenberg-MacLane space $X=K(\Q,2n)$
is g-regular since its loop space is $K(\Q, 2n-1)\simeq S^{2n-1}$. On
the other hand an odd  Eilenberg-MacLane space $X=K(Q,2n+1)$ 
is not g-regular since its loop space is $K(\Q, 2n)$ which has
homology polynomial on one generator of degree $2n$. 

Finally, since $\Omega X$ splits as a product of Eilenberg-MacLane
space for any rational space,  we see $H_*(\Omega X) $ is finite
dimensional if and only if  the homotopy groups of $X$ are finite
dimensional and in even degree. Building $X$ from its Postnikov tower
we see inductively that the cohomology of its Postnikov sections are 
entirely in even degrees (and polynomial). Since the homotopy of $X$
is in even  degrees the k-invariants are in odd degrees and hence
zero. This shows that $X$ is a finite product of even
Eilenberg-MacLane spaces. 
\end{proof}

\subsection{g-regularity for $p$-complete spaces }

Of the commutative algebra definitions, the only one with a straightforward
counterpart for $C^*(X)$ is g-regularity. 
We have seen in Subsection \ref{subsec:cochains},  
that for connected $p$-complete spaces $X$ with finite fundamental
group, if  $R=C^*(X;\Q)$ we have $\pi_*(\cE)=H_*(\Omega X)$. Accordingly a
g-regular rational space is one with $H_*(\Omega X)$ finite
dimensional. 

\begin{example} For any finite $p$-group $G$, the classifying space
$X=BG$ is $p$-complete, and hence satisfies our hypotheses without
further completion. Thus  $\Omega X$ is equivalent to the
finite discrete space $G$, so  $\cE\simeq kG$ and $X$ is g-regular. 

On the other hand the coefficient ring $H^*(BG)$ is very rarely 
regular, so this gives many examples of g-regular ring spectra which are not
c-regular.  
\end{example}

\begin{example}
Extending this idea, if $G$ is any compact Lie group with component
group a finite  $p$-group we have 
$$\Omega(BG_p^{\wedge})\simeq (\Omega BG)_p^{\wedge}\simeq
G_p^{\wedge}, $$
so that $X=(BG)_p^{\wedge}$ is again g-regular over $\Fp$. 

In fact  $X=(BG)_p^{\wedge}$ is g-regular if and only if $\pi_0(G)$ is
$p$-nilpotent,  and we will see explicit examples later where
completed classifying spaces are  not g-regular.
\end{example}

In fact this is the starting point of Dwyer and Wilkerson's project
\cite {DWpcompact} to 
capture properties of groups at a single prime homotopically. 
They  define a  connected $p$-complete space $X$ to be {\em the
classifying spaces of a connected $p$-compact group} if $H_*(\Omega X;\Fp)$
is  finite dimensional (i.e., if $C^*(X;\Fp)$ is g-regular). In other
words, a {\em $p$-compact group} is an $\Fp$-finite loop space whose
classifying space is $p$-complete and  g-regular.

A major programme involving many people has given a classification of
connected $p$-compact groups following  the lines of that for compact
Lie groups. This starts with the theorem of Dwyer and Wilkerson that connected $p$-compact groups have a 
maximal torus, and has finally culminated in the  classification of
connected $p$-compact groups in terms of  the associated $p$-adic root data
(Anderson, Grodal, M\o ller and Viruel (\cite{AGMV} for odd primes, and
\cite{AG} for $p=2$). At the prime 2 there is only one simple
2-compact group which is not obtained by completing a compact
connected Lie group. At odd primes there are many exotic examples,
because of the many $p$-adic reflection groups. It is remarkable that
objects defined purely in terms of a finiteness condition can be completely classified. 

\subsection{s-regularity for $p$-complete spaces}

Looking back at the definition of s-regularity for ring spectra, the
basic ingredient is to consider the multiplicative sequence 
$$R\lra K(x)\lra K(x)\tensor_Rk = \Lambda (\tau)$$
where $\tau$ is a generator of degree $n+1$. Since we are working with
$k$-algebras, exterior algebras are formal, so  if $n\geq 0$, we have
$\Lambda (\tau) \simeq C^*(S^{n+1})$. 
Thus if $R=C^*(X)$, the Eilenberg-Moore theorem shows that the most
obvious source of such cofibre sequence is  a spherical fibration
$$X\lla X_1 \lla S^{n+1}. $$
It is natural when talking about spaces to restrict the cofibrations
to be of this type

\begin{defn}
A space $X$ is  $s$-regular if there are $k$-complete  fibrations
$$X=X_0\lla X_1 \lla S^{n_1}, X_1\lla X_2 \lla S^{n_2}, 
\ldots ,X_{r-1}\lla X_r \lla S^{n_r}$$
with $X_r\simeq *$ and $n_i\geq 1$.  
\end{defn}

Examples of this type seem very rare, but do arise from the classical
infinite families of compact Lie groups. 

\begin{example}
$X=BU(n)$ is s-regular, in view of the fibrations
$$BU(n)\lla BU(n-1)\lla U(n)/U(n-1)\cong S^{2n-1}. $$
Similarly for $BSO(n), BSpin (n), BSp(n)$.
\end{example}

\section{Finite generation}
\label{sec:normalizable}

In conventional algebra, finite generation is one of the most
important finiteness conditions. However it is not all obvious how
best to translate this into a homotopy invariant context. In this
section we describe an approach we have found useful, which is based
on the idea that finite generation and smallness coincide over regular
rings.

\subsection{Finiteness conditions}
If $R$ is a conventional Noetherian commutative ring and $M$ is a  module
 over it,  there are a number of natural finiteness conditions. 
If $M$ has a finite  resolution by finitely generated projectives then
it is finitely built from $R$ and hence small in the derived category; the converse is also true, since one may factor the identity 
map through a finite truncation of the projective resolution. These
two equivalent homotopy invariant finiteness conditions on $M$ are
extremely useful.

On the other hand, it is not so clear how to give a homotopy invariant version of the notion of being finitely 
generated. In this section we discuss one useful method. It is based
on the fact that for a {\em regular} conventional ring  $Q$ finitely
generated  modules coincide with small modules.

\subsection{Normalizable ring spectra.}
In commutative algebra,  it is natural to assume rings are Noetherian, and  one of the most useful consequences for $k$-algebras is Noether normalization, stating that a Noetherian $k$-algebra is a
finitely generated module over a polynomial subring. 

\begin{defn}
We will say that a ring spectrum $R$ is {\em g-normalizable}, if there is a g-regular  ring spectrum $S$ with $S_*$ Noetherian,  and a ring map $S\lra R$ making $R$ into small $S$-module. In this case $S\lra R$ is called a 
{\em g-normalization} and $R\tensor_S k$ is its  {\em Noether fibre}.

If $S$ can be chosen so that its coefficient ring $S_*$ is regular, we say $R$ is {\em c-normalizable}. 
\end{defn}

As usual, with coefficient-level conditions, c-normalizations are a
bit rigid, but they do give a template for comparison. 
\begin{lemma}
\label{lem:creg}
(i) If $S_*$ is regular then an $S$-module $N$ is small if and only if $N_*$ is a finitely generated
$S_*$-module. 

(ii)  If $R$ is c-normalizable then $R_*$ is Noetherian.  
\end{lemma}

\begin{proof}
Part (ii) follows from Part (i), so we prove Part (i). 

It is easy to see that we have a cofibre sequence $X\lra Y\lra Z$ of
$S$-modules, if two terms have finitely generated homotopy, so does
the third. Thus any small module has finitely generated homotopy. 

For the converse, we start by noting that projectives are realizable. Next, if $P_*$ is
projective, $\Hom_S(P,M)=\Hom_{S_*}(P_*, M_*)$. Now for any $M$
we may choose a projective $S_*$-module mapping onto $M_*$ and realize
it by a map $P\lra M$ with fibre $F$. If $M_*$ was not projective to
start with, $F_*$ will be of lower projective dimension. Since $S_*$
is regular, this process terminates. 
\end{proof}

\begin{example} (Venkov)
If $G$ is a compact Lie group (for example a finite group),
then $C^*(BG)$ is c-normalizable. Indeed, we may 
choose a faithful representation $G \lra U(n)$, giving a fibration
$U(n)/G \lra BG \lra BU(n)$. 

It  is an unpublished consequence of work of
Castellana and Ziemianski that every $p$-compact group has a faithful linear
representation, which is to say that if
$B\Gamma$ is regular there is a map $B\Gamma \lra
BSU(n)$, for some $n$, whose homotopy fibre is $\Fp$-finite.
This means that every $p$-compact group is c-normalizable.
\end{example}

\subsection{Finitely generated modules.}
We are now equipped to give our definition. 

\begin{defn}
Suppose that $R$ is a commutative ring spectrum and $M$ is an $R$-module. 
If we are given a normalization $\nu : S \lra R$,  we say that 
 $M$ is {\em $\nu$-finitely generated} if $M$ is small over $S$. 

If $R$ is g-normalizable, we say that $M$ is {\em finitely generated} if $M$ is $\nu$-finitely generated for every
g-normalization $\nu$. 
\end{defn}

Some elementary consequences follow as in Lemma \ref{lem:creg}. 

\begin{lemma}
(i) If $R$ is g-normalizable and $M$ is finitely generated then $M_*$ is a finitely generated
$R_*$-module.

(ii)  If $\nu: S\lra R$ is a  c-normalization, then $M$ is $\nu$-finitely generated if and only if $M_*$ is finitely generated
over $R_*$.
 \end{lemma}

\begin{remark}
(a) If $R$ is c-normalizable, it is natural to refer to the class of
modules $M$ for which $M_*$ if finitely generated over $R_*$ as 
{\em c-finitely generated.}

(b) This raises the question of whether there are c-finitely generated
modules which are not finitely generated. In other words, if $R$ is a c-normalizable
g-regular ring and $M$ is an $R$-module with $M_*$ finitely generated
over $R_*$ does it follow that $M$ is small? It would suffice to show
that $R$ has a c-normalization $S$ which is an $S$-module retract of
$R$, as happens when $S_*\lra R_*$ is a monomorphism and $R_*$ is
Gorenstein. 
\end{remark}
\begin{proof}
(i) If $M$ is small over $S$ for some normalization, then $M_*$ is finitely generated over $S_*$ and hence over $R_*$. 

(ii) If $M_*$ is finitely generated over $R_*$ it is finitely generated over $S_*$. Since $S_*$ is regular, this means $M$ is small over $S$ and hence finitely generated. 
\end{proof}

\subsection{Singularity categories}
In conventional commutative algebra the Buchweitz singularity category 
$\sfD_{sing}(R)$ is defined to be the quotient of the bounded derived
category by the subcategory of perfect complexes. 

If we are given a normalization $\nu$ of a commutative ring spectrum $R$, it is evident that the category 
of $\nu$-finitely generated modules is a triangulated subcategory of
the category of $R$-modules, and similarly for the finitely generated
modules. Evidently for any particular g-normalization $\nu$  we have
full and faithful embeddings 
$$\sfD_{small}(R)\subseteq \sfD_{fg} (R) \subseteq \sfD_{\nu-fg}(R)\subseteq \sfD_{c-fg}(R)$$
of triangulated categories each closed under retracts. We could
consider any one of the quotients, but the most natural is the
initial quotient
$$\sfD_{sing}(R)=\sfD_{fg}( R)/\sfD_{small}(R) .$$
This has the merit that if $R$ is itself g-regular, the quotient is
trivial, and  in many cases one can show that $\sfD_{fg}(R)=\sfD_{c-fg}(R)$.

The singularity quotients are understood in many algebraic cases, but
for the present we will restrict ourselves to an example where we can
see from first principles that the quotient is non-trivial. See
\cite{GS} for further discussion. 

\begin{example}
If $R=C^*(BA_4)$ with $k=\Ftwo$ then $M=C^*(BV_4)$ is c-finitely
generated since $H^*(BV_4)$ is finitely generated over $H^*(BA_4)$ by
Venkov's Theorem. However if we let $F=\fib (BV_4 \lra BA_4)$ there is
a fibre sequence $\Omega S^3 \lra F \lra SO(3)/V_4$ which may be used
to check $H^*(F)$
is non-zero in infinitely many degrees. Thus $C^*(BV_4)$ is not small over $C^*(BA_4)$. 
\end{example}






\part{Hypersurface rings}
This is the second of a series of parts that take particular classes of commutative local 
rings, and identifies ways of giving homotopy invariant counterparts
of the definitions, which apply to ring spectra. 
The justification consists of the dual facts that the new definition reduces to the old in 
the classical setting and that the new definition covers and
illuminates examples in new contexts. 

Although we restrict attention to hypersurface rings here, we begin in
Section \ref{sec:ci} with a brief sketch of relevant facts about  complete intersections
in general to set  the context.  In Section \ref{sec:hypersurfaces1}
we describe the algebraic theory of  hypersurface rings in a bit more
detail. In Section \ref{sec:bimod} we explain the relevance of
bimodules and Hochschild cohomology. In view of the discussion in Part
3, it is then quite routine in Section \ref{sec:cispaces} to provide
the definitions of ci ring spectra, and we illustrate some of the
results with examples. In Sections \ref{sec:shyperiszhyper} and
\ref{sec:zhyperisghyper} we explain how the three homotopy invariant
versions of the ci condition are related to each other. 

\section{Complete intersections}
\label{sec:ci}
To understand the significance of hypersurfaces, we should first say a word about complete intersections.
There is a general framework for discussing these in the homotopy invariant context \cite{qzci, pzci}, but
there are few examples beyond the hypersurface case. The general theory is in some sense obtained by 
iterating the case of hypersurfaces, but there are a number of different ways of iterating it, and some work 
better than others. 

\subsection{Classical complete intersections}
These lectures will focus on hypersurfaces, but it is helpful to set the more general context with a
brief discussion of complete intersections in general. Just as for regular rings there are three styles
of definition: (s) a structural one (g) one involving growth and (m) one involving the homological algebra 
of modules. These all have counterparts in the homotopy invariant setting which we will introduce in 
due course. 

We start with the structural definition. 
The best behaved subvarieties of affine space are those which are specified by the right number of
equations: if they are of codimension $c$ then only $c$ equations are required.
On the same basis, a commutative local ring $R$ is a {\em structurally
  complete intersection (sci)}
if its completion is the quotient of a regular local ring $S$ by a regular sequence,
$f_1,f_2, \ldots, f_c$ (this would normally just be called `ci', but
the longer name eases comparison with the case of ring spectra). We will suppose that $R$ is complete, so that
$$R=S/(f_1,\ldots, f_c).$$
The smallest possible value of $c$ (as $S$ and the regular sequence vary)
is called the {\em codimension} of $R$. A hypersurface is the special case when 
$c=1$ so that a single equation is used.

When it comes to growth, if $R$ is ci of codimension $c$, one may construct
a resolution of any finitely generated module growing like a polynomial
of degree $c-1$. In particular the ring $\Ext_R^*(k,k)$ has polynomial growth
(we say that $R$ is {\em gci}). Perhaps the most striking result about ci rings is the theorem of Gulliksen
\cite{Gulliksen} which states that this characterises ci rings so that the
sci and gci conditions are equivalent for local rings.

With a little care, one may construct the resolutions in an eventually multi-periodic fashion: the projective
resolution eventually has the pattern of  the tensor product of $c$ periodic exact sequences. In fact
the construction is essentially independent of the module and the calculation can be phrased in terms
of the Hochschild cohomology of $R$.  This opens the way to
the theory of support varieties for modules over a ci ring \cite{AB}.

\subsection{Homotopy invariant versions, and Levi's groups}

One may give homotopy invariant versions of all three characterizations of
the ci condition:
\begin{description}
\item[(sci)] the `regular ring modulo
regular sequence' condition,
\item[(mci)] the `modules have eventually
multiperiodic resolutions' and
\item[(gci)] polynomial growth of the Ext algebra
\end{description}

 We note that the Avramov-Quillen
characterization of ci rings in terms of Andr\'e-Quillen homology
does not work for cochains on a space in the mod $p$ context since Mandell has shown \cite{Mandell}
that the topological  Andr\'e-Quillen cohomology
vanishes very generally in this case.

It  involves some work to describe the sci and mci definitions, but if we take $R=C^*(X)$ for a $p$-complete 
space Subsection \ref{subsec:cochains} shows $H_*(\Hom_{C^*(X)}(k,k))=H_*(\Omega X)$, so it is easy to understand the gci condition. 
Taking $X=(BG)_p^{\wedge}$ we may consider what this means. Of course
if $G$ is a $p$-group,  $X=BG$ so $\Omega X\simeq G$ and  $H_*(\Omega X)$ is 
simply the group ring $kG$ in degree 0 and therefore finite dimensional. More generally, it is known
to be of polynomial growth in certain cases (for instance $A_4$ or $M_{11}$
in characteristic 2) and
R.~Levi \cite{Levi1,Levi2,Levi3,Levi4}
has proved there is a dichotomy between small growth and large growth,
and given examples where the growth is exponential. Evidently groups
whose $p$-completed classifying spaces have
loop space homology that has exponential growth cannot be spherically resolvable,
so Levi's groups disproved a conjecture of F.Cohen.

\begin{example}
(a) It is amusing to consider the $p$-completed classifying space
$BG_p^{\wedge}$ for $G=(C_p\times C_p)\sdr C_3$ where $C_3$ acts via
$(1,0)\mapsto (0,1)\mapsto (-1,-1)$. 
When $p=3$, $G$ is a $p$-group so the space is g-regular. When $p=2$ the group is the alternating group $A_4$, 
which we will see below is a hypersurface. If $p\geq 5$ then Levi shows $H_*(\Omega (BG_p^{\wedge}))$ has exponential
growth by showing that $\Omega (S^5\cup_pe^6)$ is a retract. 

(b) If the reader prefers an example with a purely algebraic treatment, Benson shows in
\cite[Example 2.2]{squeezed} that the group $(C_3)^2\sdr C_2$ has a
3-completed classifying space whose loop space has exponential growth
using the methods sketched in Subsection \ref{subsec:squeezed} below.  
\end{example}

\section{Hypersurfaces in algebra}
\label{sec:hypersurfaces1}

In this section we recall some standard constructions for hypersurface algebas. We
suppose $R=S/(f)$ is a hypersurface ring, where $S$ is a regular ring and $f$ is a nonzero element of degree $d$.
Thus we have a short exact sequence
$$0 \lra \Sigma^d S \stackrel{f}\lra S \lra R \lra 0$$
of $S$-modules for a regular local ring $S$.
 There are two basic constructions that we need to generalize. 

\subsection{The degree 2 operator.}
\label{subsec:operator}
We describe a construction of a cohomological operator due to Gulliksen 
\cite{Gulliksen2,AvramovCRM}. We will do this for a single module, but 
it is apparent that the construction is essentially independent of the module, and
in fact it lifts to Hochschild cohomology. 

Given an $R$-module $M$ we may apply $(\cdot) \tensor_SM$ to the defining
sequence to obtain the short exact sequence
$$0 \lra \Tor_1^S(R,M) \lra \Sigma^d M \stackrel{f}\lra M \lra R\tensor_S M \lra 0.$$
Since $f=0$ in any $R$-module, we conclude
$$\Tor_1^S(R,M)\cong \Sigma^d M, R\tensor_S M \cong M, $$
and 
$$\Tor_i^S(R,M)=0 \mbox{ for } i \geq 2.$$
Next, choose a free $S$-module $F$ with an epimorphism to $M$, giving a 
short exact sequence 
$$0 \lra K \lra F \lra M \lra 0$$
of $S$-modules. Applying $(\cdot)\tensor_SR$, we obtain 
$$0 \lra \Tor_1^S(R,M) \lra \Kbar \lra \Fbar \lra M \lra 0.$$
Since $\Tor_1^S(R,M) \cong \Sigma^d M$, this gives an element 
$$\chi_f \in \Ext^2_R(M, \Sigma^d M), $$
or a map 
$$\chi_f: M \lra \Sigma^{d+2}M$$
in the derived category of $R$-modules. 
We sketch below how this construction lifts to give an element 
$$\chi_f \in HH^{d+2}(R|S), $$
and hence in particular that it gives a natural transformation of the
identity functor  (i.e., an element of the centre
$Z\sfD(R)$ of $\sfD(R)$).

\subsection{The eventually periodic resolution.}
\label{subsec:eventperiod}
Continuing with the above discussion, we may show that all modules $M$ have
free resolutions over $R$ which are eventually periodic of period 2. 

Indeed, $M$ has a finite free $S$-resolution  
$$0 \lra F_n \lra F_{n-1} \lra \cdots \lra F_1 \lra F_0 \lra M \lra 0. $$
Adding an extra zero term if necessary, we suppose for convenience that $n$ is even.
Now apply $(\cdot) \tensor_S R$ to obtain a complex
$$0 \lra \Fbar_n \lra \Fbar_{n-1} \lra \cdots \lra \Fbar_1 \lra \Fbar_0 \lra M \lra 0. $$
Since $\Tor_i^S(R, M)=0$ for $i \geq 2$, this is exact except in homological 
degree 1, where it is $\Tor_1^S(R,M)\cong \Sigma^d M$. Splicing in a second copy of the resolution, 
we obtain a complex
\[ \addtolength{\arraycolsep}{-0.4mm}
\begin{array}{ccccccccccccccccc}
0           &\lra &
\Fbar_n     &\lra & 
\Fbar_{n-1} &\lra & 
     \cdots &\lra &
\Fbar_2     &\lra &
\Fbar_1     &\lra &
\Fbar_0     &\lra &
M           &\lra &
0 \\
\oplus      &&
\oplus      &&
\oplus      &&
      &&
\oplus      &\nearrow&
      &&
      &&
      &&\\
\Sigma^d \Fbar_{n-1} &\lra &
\Sigma^d \Fbar_{n-2} &\lra &
\Sigma^d\Fbar_{n-3} &\lra &
     \cdots &\lra &
\Sigma^d \Fbar_0      & &    
            &       &
            &      &
           &      &
\end{array} \]
which is exact except in the second row, in homological degree 4, 
where the homology
is again $M$. We may repeatedly splice in additional rows to 
obtain a free resolution
$$\ldots \lra G_3 \lra G_2 \lra G_1 \lra G_0 \lra M \lra 0$$
over $R$. Remembering the convention that $n$ is even, provided  the degree is 
at least $n$,  the modules in the resolution are (up to suspension by a multiple of $d$)
\[ G_{2i}=\Fbar_{n} \oplus \Sigma^d\Fbar_{n-2} \oplus \cdots \oplus 
\Sigma^{d(n-2)/2}\Fbar_2 \oplus \Sigma^{dn/2}\Fbar_0 \]
in even degrees and 
\[ G_{2i+1}=\Fbar_{n-1} \oplus \Sigma^d \Fbar_{n-3} \oplus \cdots 
\oplus \Fbar_3 \oplus \Sigma^{d(n-2)/2}\Fbar_1 \]
in odd degrees. 

\subsection{Smallness.}

We may reformulate the eventual periodicity of the previous subsection in 
homotopy invariant terms. 

\begin{lemma}
If $R$ is a hypersurface $R=S/(f)$ and $M$ is a finitely generated $R$-module
 then the mapping cone of 
$\chi_f: M \lra \Sigma^{d+2}M$ is small. 
\end{lemma}

\begin{proof}
From the Yoneda interpretation, we notice that 
$$\chi_f : M \lra \Sigma^{d+2}M$$
is realized by the quotient map factoring out the first row subcomplex
$$0 \lra \Fbar_n \lra \Fbar_{n-1} \lra \cdots \lra \Fbar_1 \lra \Fbar_0\lra 0. $$
Thus the short exact sequence
$$0 \lra \Fbar_{\bullet} \lra G_{\bullet} \lra \Sigma^{d+2} G_{\bullet} \lra 0$$
of $R$-free chain complexes realizes the triangle
\begin{equation*}
\Sigma^{1}M/\chi \lra M \lra \Sigma^{d+2}M.
\qedhere
\end{equation*}
\end{proof}

\subsection{Matrix factorizations}
\newcommand{\Fpb}{\overline{F}'}
\newcommand{\Fppb}{\overline{F}''}

Because of the importance of matrix factorizations and their prominence in the IRTATCA programme, it is worth a brief subsection 
to make the connection. We have considered resolutions of $R$-modules $M$, and we note that
if $M$ is of projective dimension $1$ over $S$ then the $S$-resolution 
$$0\lla M \lla F_0\stackrel{A}\lla F_1\lla 0$$
gives a periodic resolution 
$$0\lla M \lla \Fbar_0 \lla \Fbar_1 \lla \Fbar_0 \lla \Fbar_1 \lla \Fbar_0 \lla \Fbar_1 \lla\cdots $$
over $R$. 

Since $f=0$ on $M$, multiplication by $f$ on the $S$-resolution is null-homotopic and so 
there is a diagram
$$\diagram
0& M\lto \dto^0 &F_0\lto \dto^f 
& F_1\lto_A \dto^f \dlto^B& 0\lto \\
0& M\lto         &F_0\lto         & F_1\lto^A        & 0\lto \\
\enddiagram $$
In short we have two maps $A: F_1 \lra F_0$ and $B: F_0\lra F_1$ with $AB=f\cdot id, BA=f\cdot id$. Since $F_0$ and $F_1$ are free modules, we can choose bases and represent $A$ and $B$ by matrices, so this structure 
is known as a {\em matrix factorization}. 

If we suppose $S$ is of dimension $d$,  $R$ is of dimension $d-1$. We note that 
the homological dimension of $M$ over $S$ can be expressed as an
invariant of $M$ as an $R$-module. In our case $M$ is of projective
dimension 1 over $S$ and  by the Auslander-Buchsbaum formula, it is of
depth $d-1$ as an $R$-module, which is to say it is  a maximal Cohen-Macaulay 
module. For a  module $M$ of depth $d-1-i$  (i.e., of projective dimension  $i+1$ over $S$) the above  discussion applies to its $p$th  syzygy.

\section{Bimodules and natural endomorphisms of $R$-modules}
\label{sec:bimod}

To describe the m-version of the definition of hypersurface rings, we need to briefly discuss bimodules. 

\subsection{The centre of the derived category of $R$-modules}
If $R$ is a commutative Noetherian ring and $M$ is a finitely
generated  $R$-module with an eventually $n$-periodic resolution, there
is a map $M\lra \Sigma^nM$ in the derived category whose mapping cone
is  a small $R$-module. If $R$ is an ungraded hypersurface ring, all finitely generated modules
have such resolutions with $n=2$. The lesson learnt from commutative algebra is
that to use this to characterize hypersurfaces we need to look at all
such modules $M$ together, and ask for a natural transformation
$1\lra \Sigma^n 1$ of the identity functor. By definition the centre
$Z\sfD(R)$ is the graded ring of all such natural transformations.
There are various ways of constructing elements of the centre, and
various natural ways to restrict the elements we consider.
 Some of these work better than others, and it is the purpose of this
section is to introduce these ideas.

\subsection{Bimodules}

We consider a map $S\lra R$, where $S$ is regular and $R$ is small
over $S$. We may then
consider $R^e=R\tensor_SR$, and $R^e$-modules are $(R|S)$-bimodules.
The Hochschild cohomology ring is defined by
$$HH^*(R|S)=\Ext_{R^e}^*(R,R).$$

If $f:X\lra Y $ is a map of  $(R|S)$-bimodules, for any $R$-module $M$
we obtain a map  $f\tensor 1 : X\tensor_R M\lra Y\tensor_R M$ of
(left) $R$-modules.

The simplest way for us to use this is that if we have isomorphisms
$X\cong R$ and  $Y\cong \Sigma^n R$ as $R$-bimodules, the map
$f\tensor 1: M\lra \Sigma^n M$  is natural in $M$ and therefore gives
an element of codegree $n$ in $ZD(R)$: we obtain a map of rings
$$HH^n(R) =\Hom_{R^e}(R,\Sigma^n R)\lra Z\sfD(R)^n.$$

Continuing, if $X\finbuilds_{R^e} Y$ then $X\tensor_R M\finbuilds_R
Y\tensor_R M$. In particular, if $X=R$ builds a small $R^e$-module
$Y$ then
$$ M=R\tensor_R M\finbuilds_R Y\tensor_R M \finbuiltby
R^e\tensor_RM=R\tensor_S M.$$
Thus if $M$ is finitely generated (i.e., small over $S$), this shows
$M$ finitely builds a small $R$-module.

It is sometimes useful to restrict the maps permitted in showing that $X
\finbuilds_{R^e}Y$.  If we restrict to using maps of
positive codegree coming from Hochschild cohomology, we write
$X\finbuilds_{hh} Y$.  If we permit any maps of positive codegree from the centre
$Z\sfD(R^e)$ we write $X\finbuilds_z Y$. Finally,   we could relax further and
require only that all the maps involved in building are endomorphisms
 of non-zero degree for some object and write $X\finbuilds_e Y$. It is
 this last condition that turns out to be most convenient for complete
 intersections of larger codimension.


\section{Hypersurface ring spectra}
\label{sec:cispaces}
We are now ready to describe the s-, g- and m- versions of the hypersurface condition for ring spectra. 
As usual  we begin with the template in commutative algebra, go on to describe it in general and then make it 
concrete for spaces. 

\subsection{The definition in commutative algebra}

In commutative algebra there are three styles for a definition of
a hypersurace ring: ideal theoretic, in terms of the growth of
the Ext algebra and a derived version. See \cite{cazci} for a more
complete discussion.

\begin{defn}
 (i) A local Noetherian ring $R$ is an {\em s-hypersurface}
ring if $R=S/(f)$
some regular ring $S$ and some  $f\neq 0$.

(ii) A local Noetherian ring $R$ is a  {\em g-hypersurface}
if the dimensions $\dim_k(\Ext_R^n(k,k))$ are bounded independently of $n$.

(iii) A local Noetherian ring $R$ is a  {\em z-hypersurace} if there is
an elements $z \in Z\sfD(R)$ of non-zero degree
so that
$M/z$ is small for all finitely generated modules
$M$.  Similarly $R$ is an {\em hh-hypersurface} if the element $z$ can be chosen to
come from Hochschild cohomology.
\end{defn}


\begin{thm} (\cite{AvramovCRM}, transcribed into the present language in \cite{cazci})
For a local Noetherian ring the s-, z- and g-hypersurface conditions are all  equivalent.
\end{thm}

\subsection{Definitions for ring spectra}

The appropriate definition of an s-hypersurface can be seen from the discussion of 
s-regular spectra: the point is that regular elements correspond to
maps with exterior cofibres. 

\begin{defn}
(i) A ring spectrum $R$ is a {\em c-hypersurface} if $R_*$ is a hypersurface ring. 

(ii) The ring spectrum $R$ is an {\em s-hypersurface} if there is a normalization 
$S\lra R$ with $\pi_*(R\tensor_Sk)=\Lambda_k(\tau)$. 

(iii) The ring spectrum $R$ is a  {\em g-hypersurface}
if $\dim_k(\pi_n(\Hom_R(k,k)))$ is bounded independently of $n$.

(iv) The c-normalizable ring spectrum $R$  is a  {\em z-hypersurace} if there is
an element $z \in Z\sfD(R)$ of non-zero degree
so that $M/z$ is small for all finitely generated modules
$M$.  Similarly $R$ is an {\em hh-hypersurface} if the element $z$ can be chosen to
come from Hochschild cohomology.
\end{defn}

One immediate source of examples comes from c-hypersurfaces.
\begin{example}
\label{eg:homciisgci}
If $R$ is a c-hypersurface then the spectral sequence
$$\Ext_{R_*}^{*,*}(k,k)\Rightarrow \pi_*(\Hom_R(k,k))$$
shows that it is a g-hypersurface.
\end{example}

\subsection{Definitions for spaces}
In view of the fact that regular elements correspond to 
spherical fibrations, adapting the above definitions for spaces is straightforward.

\begin{defn}
(i)  A space $X$ is a {\em c-hypersurface} if $H^*(X)$ is a
hypersurface ring. 

(ii) A space $X$ is an  {\em s-hypersurface} (or spherical hypersurface) if it is the total space of a spherical 
fibration over  a connected
g-regular space $B\Gamma$: 
there is a g-regular space $B\Gamma$ and a
fibration
$$S^{n}\lra X \lra B\Gamma$$
for some $n$.

(iii) A space  $X$ is a {\em g-hypersurface} space
if $H^*(X)$ is Noetherian and $\dim_kH_n(\Omega X)$ is bounded
independently of $n$.

(iv) A space $X$ is a {\em z-hypersurface} space if $X$ is g-normalizable and
there is an  element $z \in Z \sfD(C^*(X))$ of non-zero degree
so that $C^*(Y)/z$ is small for all finitely generated 
$C^*(Y) $. It is an {\em hh-hypersurface} if $z$ comes from Hochschild
cohomology.
(The direct transcription for ring spectra would require that this holds for all
finitely generated modules and not just those of the particular form $C^*(Y)$).
\end{defn}

\begin{remark}
Other variants have arisen, such as  $\omega$sci
 where we are permitted to  use loop spaces on spheres rather than spheres.
These conditions arose in Levi's work. This is evidently a weakening 
of sci which still implies gci.
\end{remark}

For  g-normalizable spaces $X$, we have \cite{qzci, pzci}  the implications
$$s-hypersurface \stackrel{} \Rightarrow z-hypersurface
\Rightarrow g-hypersurface  .$$
Over $\Fp$ the final implication may be reversed; the proof is given in \cite{pzci} and relies on \cite{FHT2}.

\subsection{An example}

To start with we may consider the space $BA_4$ at the prime 2. We will
observe directly that it is a hypersurface according to any one of the
definitions. 

To start with, we note that
$$H^*(BA_4)=H^*(BV_4)^{A_4/V_4}=k[x_2,y_3,z_3]/(r_6)$$
where $r_6=x_2^3+y_3^2+ y_3z_3+z_3^2$.
This shows that $BA_4$ is actually a c-hypersurface and hence also a g-hypersurface. 
Indeed,  the Eilenberg-Moore spectral sequence shows that the loop space
homology will eventually have period dividing 4. 

In fact we see that $BA_4$ is an s-hypersurface space at 2.  
The direct symmetries of a tetrahedron give
a homomorphism $A_4 \lra SO(3)$ and hence a map $BA_4 \lra BSO(3)$. The fibre is $SO(3)/A_4$, and at the
prime 2 this is $S^3$, so there is a 2-adic fibration 
$$S^3\lra BA_4 \lra BSO(3).$$

In Subsection \ref{subsec:A4squeezed}, we will also describe a
representation theoretic approach to showing that $BA_4$ is a g-hypersurface, which 
will show that its ultimate  period is exactly 4.

\subsection{Squeezed homology.}
\label{subsec:squeezed}

Since we are working with groups, it is
illuminating to recall Benson's purely representation
theoretic calculation of the loop space homology $H_*(\Omega (BG_p^{\wedge}))$
\cite{squeezed}. In fact he defines the {\em squeezed homology}
 groups $H\Omega_*(G;k) $ algebraically  and proves
$$H_*(\Omega(BG_p^{\wedge}))\cong H\Omega_*(G;k), $$

In more detail, $ H\Omega_*(G;k)$ is the homology of
$$\cdots \lra P_3 \lra P_2 \lra P_1 \lra P_0, $$
 a so-called {\em squeezed resolution}
of $k$. The sequence  of projective $kG$-modules $P_i$
is defined recursively as follows. To
start with $P_0=P(k)$ is the projective cover of $k$.
Now if $P_i$ has been constructed, take $N_i = \ker (P_i \lra P_{i-1})$
(where we take $P_{-1}=k$), and $M_i$ to be the smallest submodule of
$N_i$ so that $N_i/M_i$ is an iterated extension of copies of $k$.
Now take $P_{i+1}$ to be the projective cover of $M_i$.

\subsection{Trivial cases}
Note that if $G$ is a $p$-group, we have $\Omega(BG_p^{\wedge})\simeq G$
so that $H_*(\Omega BG) \cong kG$ and since
$k$ is the only simple module, $M_0=0$ and we again find $H\Omega_*(G)=kG$.

We would expect the next best behaviour to be when $H^*(BG)$ is a hypersurface.
Indeed, if $H^*(BG)$ is a polynomial ring modulo a relation of codegree $d$,
the Eilenberg-Moore spectral sequence
$$\Ext_{H^*(BG)}^{*,*}(k,k) \Rightarrow H_*(\Omega(BG_p^{\wedge}))$$
shows that there is an ultimate periodicity of period $d-2$. The actual
period therefore divides $d-2$. We now return to the example of $A_4$.

\subsection{$A_4$ revisited}
\label{subsec:A4squeezed}
We described the homotopy theoretic proof that $BA_4$ is
an s-hypersurface space above. 

Here we sketch a purely  algebraic proof from \cite{squeezed}. To start with, we would like to see algebraically that
$H_*(\Omega (BA_4)_2^{\wedge})$ is eventually periodic. 

This case is small enough to be able
to compute products in $H_*(\Omega BG_p^{\wedge},k)$ using squeezed resolutions,
and we get
\[ H_*(\Omega BG_p^{\wedge},k)=\Lambda(\alpha)\otimes
k\langle\beta,\gamma\rangle/(\beta^2,\gamma^2) \]
with $|\alpha|=1$ and $|\beta|=|\gamma|=2$.
Beware that $\beta$ and $\gamma$ do not commute, so that a $k$-basis
for $H_*(\Omega BG_p^{\wedge},k)$ is given by alternating
words in $\beta$ and $\gamma$ (such as $\beta\gamma\beta$ or the empty word),
and $\alpha$ times these alternating words.

There are three simple modules. Indeed, the quotient of $A_4$ by its
normal Sylow 2-subgroup is of order 3; supposing for simplicity that $k$
contains three cube roots of unity $1, \omega, \overline{\omega}$, the simples
correspond to how a chosen generator acts. The projective covers of the
three simple modules are
\[ P(k)=\vcenter{\xymatrix@=2mm{&k\ar@{-}[dl]\ar@{-}[dr]\\
\omega\ar@{-}[dr]&&\bar\omega\ar@{-}[dl]\\&k}},\quad
P(\omega)=\vcenter{\xymatrix@=2mm{&\omega\ar@{-}[dl]\ar@{-}[dr]\\
\bar\omega\ar@{-}[dr]&&k\ar@{-}[dl]\\&\omega}},\quad
P(\bar\omega)=\vcenter{\xymatrix@=2mm{&\bar\omega
\ar@{-}[dl]\ar@{-}[dr]\\k\ar@{-}[dr]&&\omega
\ar@{-}[dl]\\&\bar\omega}}. \]
With this information to hand, the reader may  verify the assertion
about the loop space homology ring by using squeezed resolutions.

\section{s-hypersurface spaces and z-hypersurface spaces}
\label{sec:shyperiszhyper}

In the algebraic setting the remarkable fact is that modules
over hypersurfaces have eventually periodic resolutions, and hence that
they are hhci of codimension 1. The purpose of this section is to sketch 
a similar result for spaces. The argument is given in detail in \cite{pzci}.

\begin{thm}
\label{thm:shyperiszhyper}
If $X$ is an s-hypersurface space with fibre sphere of dimension $\geq 2$
then $X$ is a z-hypersurface space.
\end{thm}


\subsection{Split spherical fibrations.}
\label{subsec:ssf}
The key in algebra was to consider  bimodules, for which we consider the
(multiplicative) exact sequence
$$R \lra R^e \lra R^e \tensor_{R} k , $$
where the first map is a monomorphism split by the map $\mu$ along which $R$ acquires
its structure as an $R^e$-module structure.
This corresponds to the pullback fibration
$$X \lla X\times_{B\Gamma}X \lla S^n, $$
split by the diagonal
$$\Delta : X \lra X \times_{B\Gamma }X$$
along which the cochains on $X$ becomes a bimodule.
To simplify notation, we consider a more general situation: a fibration
$$B \lla E \lla S^n$$
with section $s: B \lra E$. The case of immediate interest  is $B=X$,
$E=X\times_{B\Gamma}X$, where a $C^*(E)$-module is a $C^*(X)$-bimodule.

Note that by the third isomorphism theorem for fibrations, there is
a fibration
$$\Omega S^n \lra B \stackrel{s}\lra E. $$
This gives the required
input for the following theorem. The strength of the result is that the
cofibre sequences are of $C^*(E)$-modules.

\begin{thm}
\label{thm:loopspherebuild}
Suppose given a fibration $\Omega S^n \lra B \stackrel{s}\lra E$
with $n\geq 2$.

(i) If $n$ is odd,
then there is a cofibre sequence of $C^*(E)$-modules
$$\Sigma_{n-1} C^*(B) \lla C^*(B) \lla C^*(E). $$

(ii)  If $n$ is even,
then there are cofibre sequences of $C^*(E)$-modules
$$C \lla C^*(B) \lla C^*(E) $$
and
$$\Sigma_{2n-2} C^*(B) \lla C \lla \Sigma_{n-1}C^*(E). $$
In particular the fibre of the composite
$$C^*(B)\lra C \lra \Sigma_{2n-2}C^*(B)$$
is a small $C^*(E)$-module constructed with one cell in codegree
0 and one in codegree $n-1$.
\end{thm}

\begin{remark}
Note that in either case we obtain a cofibre sequence
$$ K\lla C^*(B) \lla \Sigma_a C^*(B)$$
of $C^*(E)$-modules with $K$ small.
\end{remark}

The strategy is to first prove the counterparts
in cohomology by looking at the Serre spectral sequence
of the fibration from Part (i) and then lift the conclusion to
the level of cochains.

\section{Growth of z-hypersurface resolutions}
\label{sec:zhyperisghyper}
In this brief section we prove perhaps the simplest implication between the hypersurface
conditions: a z-hypersurface is a g-hypersurface.

\begin{lemma}
If $R$ is a z-hypersurface and $k$ is a field then the vector spaces $\pi_n(\cE)$ are of
bounded dimension and $R$ is a g-hypersurface.
\end{lemma}

\begin{proof}
Since $R$ is a z-hypersurface, then in particular there is a projective resolution of $k$ which is 
eventually periodic. In other words there is a triangle
$$\Sigma^{n}k\lra k \lra L$$ 
with $n\neq 0$ and $L$ small over $R$. Applying $\Hom_R(\cdot , k)$ we find a triangle
$$\Sigma^{-n}\cE \lla \cE \lla \Hom_R(L,k). $$
Since $L$ is finitely built from $R$, $\Hom_R(L,k)$ is finite dimensional over $k$, and hence only 
nonzero in a finite range of degrees (say $[-N,N]$). Outside that range we
have $\pi_{s+n}(\cE)\cong \pi_s(\cE)$, so every homotopy group is isomorphic as a $k$-vector 
space to one in the range $[-N-n, N+n]$ and the bound is the largest of these dimensions. 
\end{proof}

\begin{remark}
Essentially the same argument shows that a cofibre sequence $\Sigma^n A \lra A \lra B$
with two terms the same means that the growth rate of $A$ is at most one more than that of $B$.
\end{remark}

\part{Gorenstein rings}
This is the final part taking a  particular classes of commutative local 
rings, and  giving homotopy invariant counterparts of the definitions. 
The justification consists of the dual facts that the new definition reduces to the old in 
the classical setting and that the new definition covers and illuminates new examples.The generalization of the Gorenstein condition is  perhaps the most
successful of the three,  since there are so many Gorenstein 
ring spectra, and this approach provides consequences that are both unexpected and 
very concrete. 

In Section \ref{sec:Gor} we introduce the homotopy invariant version
of the Gorenstein condition an in Section \ref{sec:GorD} the
associated Gorenstein duality property and its implications for
coefficients. We then provide some basic tools for dealing with the
Gorenstein condition: ascent and descent in Section
\ref{sec:Gorascdesc} and Morita invariance in Section
\ref{sec:MoritaGorenstein}. We then turn to examples in earnest, with
a discussion of Gorenstein duality for group cohomology in Section
\ref{sec:GorHBG}, for rational homotopy theory in Section
\ref{sec:GorQ} and a brief pointer to other examples in Section
\ref{sec:Goregs}.

\section{The Gorenstein condition}
\label{sec:Gor}
In this section we quickly recall the definition of a Gorenstein local
ring in a form which also makes sense for ring spectra with a map to $k$.

\subsection{Gorenstein local rings}
The usual definition of a Gorenstein local ring is that $R$ is of finite injective dimension over itself. But
it is then proved that in fact it is then of injective dimension equal
to $r= \mbox{Krull dimension}(R)$ and
that $\Ext^*_R(k,R)=\Ext^r_R(k,R)=k$. Conversely, if this holds, the ring is Gorenstein. There are various
other characterizations of Gorenstein rings, including a duality statement that we will discuss
shortly, but this is enough to suggest the definition for ring spectra.

\subsection{Gorenstein ring spectra}
\label{subsec:Gor}
Ultimately, we want to consider  duality phenomena modelled on those in commutative
algebra of Gorenstein local rings, so we will develop the theory for
spectra in parallel. Corresponding to the Noetherian condition we restrict the class 
of ring spectra to those which are proxy-small in the sense of
Definition \ref{defn:proxysmall}.  We begin with the core Gorenstein 
condition and move onto duality  in due course.  These definitions come
from \cite{DGI}. 

\subsection{The Gorenstein condition}
We  say that   $R\lra k$ is {\em  Gorenstein} of shift
$a$ (and write $\shift (R)=a$) if we have an equivalence 
$$\Hom_R(k, R)\simeq \Sigma^a k $$
of $R$-modules.

\begin{remark}
We will say that it is c-Gorenstein if $R_*\lra k$ is a Gorenstein local ring. As usual the spectral 
sequence
$$\Ext^{*,*}_{R_*}(k,R_*)\Rightarrow \pi_*(\Hom_R(k,R))$$ 
shows that a c-Gorenstein ring spectrum is Gorenstein, but we will give many examples of Gorenstein 
ring spectra which are not c-Gorenstein. 
\end{remark}



\section{Gorenstein duality}
\label{sec:GorD}
Although the Gorenstein condition itself is convenient to work with, the real reason for
considering it is the duality property that it implies. To formulate
this, we use local cohomology in the sense of Grothendieck, and the
reader may wish to refer to Appendix \ref{sec:loccoh} for the basic
definitions.  

\subsection{Classical Gorenstein duality}
In classical local commutative algebra, the Gorenstein duality property is
that all local cohomology is in a single cohomological degree, where it is the
injective hull $I(k)$  of the residue
field. To give a formula, we write $\Gamma_{\fm}M$ for the $\fm$-power torsion in an
$R$-module $M$, and $H^*_{\fm}(M)$ for the local cohomology of
$M$, recalling Grothendieck's theorem that if $R$ is Noetherian,
$H^*_{\fm}(M)=\R^*\Gamma_{\fm}(M)$. 
The Gorenstein duality statement  for a local ring of Krull dimension $r$ therefore states
$$H^*_{\frak{m}}(R)=H^r_{\frak{m}}(R)=I(k). $$
If $R$ is a $k$-algebra, $I(k)=R^{\vee}=\Gamma_{\fm} \Hom_k(R,k)$.

\subsection{Gorenstein duality for $k$-algebra spectra}
Turning to ring spectra, we will  treat the case that  $R$ is a $k$-algebra. This simplifies
things considerably, and covers many interesting examples. The more general case requires 
a discussion of Matlis lifts, for which we refer the reader to \cite{DGI}.

In the case of $k$-algebras,  we may again 
define  $R^{\vee}=\cell_k(\Hom_k(R, k))$ and observe this has the Matlis lifting property
$$\Hom_R(T, R^{\vee})\simeq \Hom_k(T,k) $$
for any $T$ built from $k$. 

In particular, if $R$ is Gorenstein of shift $a$ we have equivalences
of $R$-modules
$$\Hom_R(k, \cell_kR)\simeq \Hom_R(k,R)\stackrel{(G)}\simeq \Sigma^a
k\stackrel{(M)}\simeq \Hom_R(k, \Sigma^a R^{\vee}),  $$
where the equivalence (G) is the Gorenstein property and the
equivalence (M) is the Matlis lifting property. 
We would like to remove the $\Hom_R(k, \cdot)$ to deduce
$$\cell_kR\simeq \Sigma^a R^{\vee}. $$
Morita theory  (specifically Lemma \ref{proxycell}) says that if $R$ is proxy-regular we may make this
deduction provided $R$ is orientably Gorenstein in the sense that the right
actions of $\cE=\Hom_R(k,k)$  on $\Sigma^ak$ implied by the two
equivalences (G) and (M) agree. 

We note that $\cell_k(R)$  is a covariant functor of $R$ whilst $R^{\vee}$ is a contravariant functor of
$R$, so an equivalence between them is a form of duality: when it
holds we say $R$ satisfies {\em Gorenstein duality}.

\subsection{Automatic orientability} 

There are a number of important cases where orientability is
automatic because $\cE$ has a unique action on $k$, and in this case
the Gorenstein condition automatically implies Gorenstein duality. 

The first case of this is when  $R$ is a classical commutative local
ring, although of course we knew already that  in this case 
the Gorenstein condition is equivalent to Gorenstein duality. 

From our present point of view, we see this as a consequence of
connectivity:  $\cE$ (whose homology is
$\Ext_R^*(k,k)$) has a unique action on $k$.
The same argument applies when the ring spectrum is both a $k$-algebra and connected.

\begin{prop}
\label{prop:autoGor}
Suppose $R$ is a proxy-regular, connected $k$-algebra
and $\pi_*(R)$ is Noetherian with $\pi_0(R)=k$ and 
maximal ideal $\fm$ of positive degree elements.  If $R$ is 
Gorenstein of shift $a$, then it is automatically orientable and so 
has Gorenstein duality. 
\end{prop}

\begin{proof} 
First we argue that if $R$ is Gorenstein, it is automatically
orientable. Indeed, we show that $\cE$ has a
unique action on $k$. Since $R$ is a $k$-algebra, the
action of $\cE$ on $k$ factors through 
$$\cE=\Hom_R(k,k) \lra \Hom_k(k,k)=k,$$
so since $k$ is an Eilenberg-MacLane spectrum,  
the action is through $\pi_0(\cE)$. Now we observe that since $R$ is connected, 
$\Ext_{R_*}^s(k,k)$ is in degrees $\leq -s$, so that the spectral
sequence for calculating $\pi_*(\Hom_R(k,k))$ shows $\cE$ is
coconnective with $\pi_0(\cE)=k$ which must act trivially on $k$. 
\end{proof}

We may go a little further to the nilpotent case. 

\begin{lemma}
If $X$ is connected with $\pi_1(X)$ a finite $p$-group and $k$ is of
characteristic $p$ then  if $C^*(X)$ is Gorenstein it automatically
has Gorenstein duality.  
\end{lemma}
\begin{proof}
Again we find $\cE$ has a unique action on $k$. Since $\cE$ is a
$k$-algebra, it acts through $\pi_0(\cE)=H_0(\Omega
X)=k[\pi_1(X)]$. By the characteristic assumption, this has a unique
action on $k$.
\end{proof}

\subsection{The local cohomology theorem}

In many cases (see Remark \ref{rk:cellular} of Appendix Dy) one can give an algebraic description of the
$k$-cellularization and infer algebraic consequences of
Gorenstein duality.  For simplicity we restrict to local $k$-algebras,
although the methods apply more generally.

\begin{lemma}
\label{lem:loccohss}
If $R_*$ is a  Noetherian $k$-algebra, with maximal ideal $\fm$ and residue field
$k$, then $k$-cellularization coincides with 
the derived $\fm$-power torsion functor.    Accordingly, if $R$ has Gorenstein duality, there is a local cohomology spectral sequence
$$H^*_{\fm}(R_*)\Rightarrow \Sigma^a R_*^{\vee}. $$
\end{lemma}

If $R$ has Gorenstein duality, Lemma \ref{lem:loccohss}  shows that the ring
$\pi_*(R)$ has very special properties (even if it falls short of
being Gorenstein), studied in \cite{GL}.  Some of these properties
were first observed by Benson and Carlson \cite{BC1,BC2} for group cohomology (corresponding to the special
case of the ring spectrum $R=C^*(BG)$, which we will see below has
Gorenstein duality).

To start with, we note that the spectral sequence collapses if $R_*$ is Cohen-Macaulay
to show $\lc^r(R_*)\cong \Sigma^{a+r} R_*^{\vee}$ (where $r$ is the Krull
dimension of $R_*$).  Thus the coefficient ring $R_*$ is also Gorenstein. 

The spectral sequence also collapses if $R_*$ is of Cohen-Macaulay defect 1, 
to give an exact sequence
$$
0 \lra \lc^r(R_*)\lra 
\Sigma^{a+r} R_*^{\vee}\lra \Sigma \lc^{r-1}(R_*)\lra 0.$$

In general,  local duality lets one deduce that  the cohomology ring $R_*$ is 
always generically Gorenstein. 

The collapse of the local cohomology theorem in the case of
Cohen-Macaulay defect $\leq 1$ has very  concrete consequences in that the 
 Hilbert series of $R_*$ satisfies a  suitable pair of  functional equations.

\begin{cor} \cite{ringlct}
\label{cor:functeqn}
Suppose $R$ has Gorenstein duality of shift $a$, that $\pi_*(R)$ is
Noetherian of Krull dimension $r$ and Hilbert series $p(s)=\sum_i\dim_k(R_i)s^i$.

If $\pi_*(R) $ is Cohen-Macaulay it is also Gorenstein, and the
Hilbert series satisfies 
$$p(1/s)=(-1)^rs^{r-a}p(s).$$

If $\pi_*(R) $ is almost Cohen-Macaulay it is also almost Gorenstein, and the
Hilbert series satisfies 
$$p(1/s)-(-1)^rs^{r-a}p(s)=(-1)^{r-1}(1+s)q(s) \mbox{ and } q(1/s)=(-1)^{r-1}s^{a-r+1}q(s).$$ 

In any case $\pi_*(R)$ is Gorenstein in codimension 0 and almost
Gorenstein in codimension 1. 
\end{cor}
 
\section{Ascent, descent and arithmetic of shifts}
\label{sec:Gorascdesc}

Very commonly we have a a map $\theta: S\lra R$ of ring spectra, and
we wish to relate properties of the two rings. With language  from the
geometric counterpart, a theorem stating that
if $R$ is Gorenstein then $S$ is Gorenstein is called a {\em descent
theorem} and a  theorem stating that
if $S$ is Gorenstein then $R$ is Gorenstein is called an {\em ascent
  theorem}. Typically the hypotheses are either about relative
properties of $\theta$ or in terms of the cofibre of $\theta$.
 
\subsection{Relatively Gorenstein maps}

We say that $\theta$ is {\em relatively Gorenstein} of shift $a$ if 
$$\Hom_S(R,S)\simeq \Sigma^a R, $$
and then write $a=\shift(R|S)$. This is quite a strong version of the
condition, since we have asked for a single untwisted ineger
suspension. 

We make the  elementary observation that for any ring map $\theta : S\lra R$ 
$$\Hom_R(k,\Hom_S(R,S))\simeq \Hom_R(R\tensor_S k, S)\simeq 
\Hom_S(k,S).$$
Thus we conclude that if $S\lra R$ is relatively Gorenstein then $R$ is
Gorenstein  if and only if  $S$ is Gorenstein, and in that case 
$$\shift (S)=\shift (R)+\shift (R|S).$$

\subsection{Two chromatic examples}
There are a number of examples where known facts amount to proving
that a map of ring spectra is relatively Gorenstein, and the resulting
descent theorems prove things of great interest. In these examples one 
shows that $\theta$ is relatively Gorenstein and $R$ is
c-Gorenstein; we then reach  the interesting conclusion that $S$ is Gorenstein.

\begin{example}
\label{eg:relGor} {\em (Complexification in $K$-theory)}
(i)  Periodic complex $K$-theory is represented by the ring spectrum
$KU$ and periodic real $K$-theory by $KO$. The coefficient ring of complex
$K$-theory is $KU_*=\Z [v,v^{-1}]$ where $v$ is of degree
$2$. Complexification gives a ring map 
$$S=KO\lra KU=R$$
and we observe that it is relatively Gorenstein. Indeed, Wood's
theorem states $KO\sm \C P^2 \simeq \Sigma^2 KU$. Since 
$$\C P^2=S^2\cup_{\eta}e^4$$
we see that there is a cofibre sequence
$$\Sigma KO\stackrel{\eta} \lra KO \lra KU $$
of $KO$-modules.  It follows by applying $\Hom_{KO}(\cdot, KO)$ that
$$\Hom_{KO}(KU,KO)\simeq \Sigma^{-2}KU$$
so that $KO\lra KU$ is relatively Gorenstein of shift $-2$.

(ii) We may take connective covers to obtain a ring map
$$S=ko\lra ku=R. $$
Now $ku_*=\Z [v]$ and killing homotopy groups gives a ring map 
$$ku\lra \Z$$
which is evidently c-Gorenstein (and hence Gorenstein) of shift $-3$. The connective version of Wood's
theorem gives a cofibre sequence
$$\Sigma ko\stackrel{\eta} \lra ko \lra ku $$
of $ko$-modules.  It follows by applying $\Hom_{ko}(\cdot, ko)$ that
$$\Hom_{ko}(ku,ko)\simeq \Sigma^{-2}ku$$
so that $ko\lra ku$ is also relatively Gorenstein of shift $-2$. Hence
we deduce that $ko \lra \Z$ is Gorenstein of shift $-5$. We note that
the ring $ko_*=\Z [\eta_1, \alpha_4, \beta_8]/(\eta^3, \eta \alpha,
\alpha^2=4\beta, 2\eta)$ is fairly complicated, and it is easy to check that $ko\lra \Z$ is not
c-Gorenstein. 

In this discussion we have taken $k=\Z$, which is not a field. If the
reader prefers, we can instead work over the field $\Fp$ and  use the fact that for any prime $p$ the map
$ku\lra \Fp$ is c-Gorenstein of shift $-4$. Since $ko\lra ku$ is
relatively Gorenstein of shift $-2$ we conclude $ko\lra \Fp$ is
Gorenstein of shift $-6$. \qqed
\end{example}

\begin{example} {\em (Topological modular forms)} Precisely similar statements hold for the ring spectrum $tmf$ of
topological modular forms at various primes. This uses results of
Hopkins-Mahowald  as proved by Matthew \cite{AM} (see
\cite{THHGor} for a slightly expanded discussion). 

(3) Localized at the prime 3, there is a map $tmf \lra tmf_0(2)$ to the ring
spectrum of topological modular forms with the indicated level structure. The
counterpart of Wood's theorem is the fact that  $tmf_0(2)\simeq tmf \sm (S^0\cup_{\alpha_1} e^{4}
\cup_{\alpha_1} e^{8})$, so that 
$$\Hom_{tmf}(tmf_0(2),tmf )\simeq \Sigma^{-8} tmf_0(2), $$
and the map is relatively Gorenstein of shift $-8$.
Since $tmf_0(2)_*=\Z_{(3)}[c_2, c_4] $ (where $|c_i|=2i$) we see that
$tmf_0(2)\lra \Z_{(3)}$ is c-Gorenstein (and hence Gorenstein) of shift
$-14$.  Hence we deduce by Gorenstein descent  that $tmf\lra \Z_{(3)} $ is
Gorenstein of shift $-22$. 

(2) Localized at the prime 2, there is a map $tmf \lra tmf_1(3)$ to
the ring spectrum of topological modular forms with the indicated level structure. The
 Here $tmf_1(3)$ is
a form of $BP\langle 2\rangle$ and  $tmf_1(3)\simeq tmf \sm DA(1)$ so that 
$$\Hom_{tmf}(tmf_1(3),tmf )\simeq \Sigma^{-12} tmf_1(3), $$
and the map is relatively Gorenstein of shift $-12$.
Since $tmf_1(3)_*= \Z_{(2)}[\alpha_1, \alpha_3]$ (where
$|\alpha_i|=2i$) we see that $tmf_1(3)\lra \Z_{(2)}$ is c-Gorenstein
(and hence Gorenstein) of shift
$-10$.  Hence we deduce by Gorenstein descent  that $tmf \lra \Z_{(2)}$ is
Gorenstein of shift $-22$. \qqed
\end{example}

Abstracting this slightly, we have 
$$\bbS \lra S \stackrel{\theta}\lra R \lra k$$
and these examples were all very special in that there was an equivalence 
$$R\simeq S\tensor_{\bbS} L$$
for a self-dual finite complex $L$. This means in turn that the
cofibre of $\theta : S \lra R$ 
$$R\tensor_S k\simeq k\tensor_{\bbS}L, $$
with homotopy $H_*(L;k)$ a Poincar\'e duality algebra. In the next
subsection we show that this weaker condition is often sufficient to give an ascent
theorem. 

\subsection{Gorenstein Ascent}

We suppose that $S\lra R \lra Q$ is a cofibre sequence of commutative
algebras with a map to $k$, and we now consider the Gorenstein ascent
question.  When does the fact that $S$ is Gorenstein imply that $R$ is
Gorenstein? It is natural to assume that $Q$ is Gorenstein, but it is
known this is not generally sufficient.

In effect the Gorenstein Ascent theorem will state that under suitable
hypotheses  (see Section \ref{subsec:Gorascent}) there is an equivalence
$$\Hom_R(k,R)\simeq \Hom_{Q}(k, \Hom_S(k,S)\tensor_k Q). $$
When this holds, it follows that if $S$ and $Q$ are Gorenstein, so is
$R$ and 
$$\shift (R)=\shift (S)+\shift (Q).$$

\subsection{Arithmetic of shifts}
We summarize the behaviour of Gorenstein shifts  in the ideal
situation when ascent and descent both hold. If all rings and maps are Gorenstein of the indicated
shifts
$$\stackrel{s} S \stackrel{\mu}\lra 
\stackrel{r} R \stackrel{\lambda}\lra 
\stackrel{q} Q$$
then $r=s+q, \mu =-q$ and $\lambda=s$

\subsection{When does Gorenstein ascent hold?}
\label{subsec:Gorascent}

The core of our results about ascent come  from \cite{DGI}. Indeed, 
the proof of \cite[8.6]{DGI} gives a sufficient condition for Gorenstein
ascent in the commutative context. 

\begin{lemma}
\label{lem:Gorascent}
If $S$ and $R$ are commutative and the natural map $\nu:
\Hom_S(k,S)\tensor_S R\lra  \Hom_S(k,R)$ 
is an equivalence then 
$$\Hom_R(k,R)\simeq \Hom_{Q}(k, Hom_S(k,S)\tensor_k Q). $$
In this case, if  $S$ and $Q$
are Gorenstein, so is $R$, and the shifts add up: $\shift (R)=\shift
(S)+\shift (Q)$. \qqed
\end{lemma}

Now that we have a sufficient condition for Gorenstein ascent, we want
to identify cases in which  it is satisfied. The map 
$$\nu: \Hom_S(k,S)\tensor_S M\lra  \Hom_S(k,M), $$
is clearly an equivalence when $M=R$ and hence for any module finitely
built from $R$. This shows that the hypotheses are satisfied 
when $R$ is small  over $S$ (or equivalently, when $Q$ is finitely
built from $k$) so that ascent holds in this case. This is already a very
useful result. 

\begin{example}
\label{eg:fibration}
If we have a fibration $F\lra E\lra B$ of
spaces to which the Eilenberg-Moore theorem applies in the sense that 
$C^*(F)=C^*(E)\tensor_{C^*(B)}k$, then if $B$ is Gorenstein and $F$ is
an orientable manifold, it follows from Gorenstein Ascent that $E$ is Gorenstein.
\end{example}

There are other important cases where $\nu$ is han equivalence. Indeed,
we can exploit the fact that the hypothesis on $\nu$ in Lemma \ref{lem:Gorascent}
only depends on $R$ as a {\em module} over $S$ to show that $\nu$ is
an equivalence  when  $R$ is suitably approximated as an inverse limit.
This is a central ingredient in proving Gorenstein duality for many
topological Hochschild homology spectra \cite{THHGor}.

\subsection{Local duality}

To start with, we clarify terminology.  The Gorenstein duality
property we have been discussing is a rare and special
thing. On  the other hand, local duality (as in the title of this
subsection) is a tool available very generally. Local duality is based on Noether
normalization,  which means that  every well behaved local ring $R$ is finite as a module over a
Gorenstein ring $S$. Local duality is the property inherited by $R$ as
a consequence of the existence of $S$.

Accordingly, we assume for the rest of this section that we are given a map 
 $S\lra R$ so that $R$ is a small $S$-module. Notice that this means
 that for an $R$-module $M$, its $k$-cellularization as an $R$-module 
is its $k$-cellularization as an $S$-module and similarly for  
$k$-completions. This is reflected in our notation.  We also note that 
$R^{\vee}=\Hom_S(R, S^{\vee})$ so that the Matlis lifts are related by
coextension of scalars.

Traditionally, local duality is thought of as saying that Matlis dual
of local cohomology (embodied in $\Hom_R(\cell_kR, R^{\vee})\simeq
(R^{\vee})_k^{\wedge}$) is isomorphic to a completed Ext group (embodied by
$\Hom_S(R,S)_k^{\wedge}$). Directly translating this into our context we reach
a more inscrutable definition. 

\begin{defn} 
\label{defn:localduality}
We say that $S\lra R$ has {\em local duality of shift $b$}  if there is an equivalence of $R$-modules
$$\Hom_S(R,S)_k^{\wedge}=\Sigma^b ( R^{\vee})_k^{\wedge}.$$
\end{defn}

We now have two properties the map $S\lra R$ may or may not have: the
relative Gorenstein property and local duality. We also have
Gorenstein duality for $S\lra k$ and for $R\lra k$. It is valuable to
disentangle the relationships
between them.

In fact we have three complete $R$-modules.

\begin{itemize}
\item $R_k^{\wedge}$
\item $(R^{\vee})_k^{\wedge}$
\item $\Hom_S(R,S)_k^{\wedge}$
\end{itemize}
The equivalence of each of the possible pairs has a name.
In the following diagram, the
label ``relGor'' means ``$k$-completion of relatively Gorenstein'', ``locD'' means
``local duality'',  and ``GorD'' means ``Gorenstein duality'';  
the superscripts indicate the suspension necessary to get from the
tail of the arrow to the head of the arrow.

$$\diagram
&\Hom_S(R,S)_k^{\wedge} \ar@{->}[dl]_{\mathrm{relGor}^c}
\ar@{->}[dr]^{\mathrm{locD}^b} & & S_k^{\wedge} \ar@{->}[dr]^{S \mathrm{GorD}^b} &\\
R_k^{\wedge}\ar@{->}[rr]_{R \mathrm{GorD}^a}&&(R^{\vee})_k^{\wedge}&&(S^{\vee})_k^{\wedge}
\enddiagram$$
This makes it clear that any two of relGor, locD and $R$GorD implies
the third, and that $a+c=b$. 

It remains to observe that locD follows from $S$GorD by coextension of
scalars from $S$-modules to $R$-modules.

\begin{lemma}
If $S$ has Gorenstein duality of shift $b$ then 
$S\lra R$ has local duality of shift $b$.
\end{lemma}

\begin{proof}
We apply $\Hom_S(R,\cdot)$ to the equivalence
$S_k^{\wedge}\simeq \Sigma^b (S^{\vee})_k^{\wedge}$
and then use the following lemma to move the completions to the outside.
\end{proof}

\begin{lemma}
\label{lem:compcores}
For any $S$-module $N$, we have an equivalence
$$\Hom_S(R,  N_{k}^{\wedge})\simeq \Hom_S(R, N)_k^{\wedge}$$
\end{lemma}

\begin{proof} 
This is a formality:
$$\Hom_S(R, \Hom_S(\cell_kS, N))
\simeq  \Hom_S(R \tensor_S\cell_kS, N)) \simeq
\Hom_R(\cell_kR, \Hom_S(R,N))$$
\end{proof}

This approach can be used to show that Gorenstein duality localizes
\cite{kappaI}.

\section{Morita invariance of the Gorenstein condition} 
\label{sec:MoritaGorenstein}

We show that the Gorenstein condition is Morita invariant 
in many useful cases, provided $R$ is a $k$-algebra. This allows us to 
deduce striking consequences from well-known examples of Gorenstein rings. 
For instance we can deduce the local cohomology theorem for finite $p$-groups 
from the fact that $kG$ is a Frobenius algebra.

\begin{thm}
\label{REGorenstein}
Suppose $R$ is a $k$-algebra, and that  $\cE$ and $R$ are Matlis reflexive.
Then
$$\HomE (k , \cE )\simeq \HomR (k,R),$$
and hence
$$\cE \mbox{ is Gorenstein } \Longleftrightarrow  R \mbox{ is Gorenstein. }$$ 
\end{thm}

\begin{proof} We use the fact that (in the notation of Section \ref{sec:Morita})
$E(R^{\vee})=\kvee $, so that by Lemma \ref{proxycell} we have 
$$R^{\vee}=TE(R^{\vee})=T\kvee =\kvee \tensorE k . $$
We also note that 
$$\Hom_R(k\tensor_R k^{\vee}, k)\simeq \Hom_k(k,\Hom_R(k,k))\simeq
\cE$$ 
so that
$$\cE^{\vee}=k  \tensorR \kvee . $$ 

Next, note that   the expression $k \tensorR \kvee \tensorE k$
makes sense, where the right $\cE$-module structure on the first two factors comes from $\kvee$.  The key equality in the proof is simply the 
associativity isomorphism
$$\cE^{\vee} \tensorE k=
   k \tensorR \kvee \tensorE k=k \tensorR R^{\vee} .$$

Now we make the following calculation,

$$\begin{array}{rcl}
\HomE(k,\cE )   &\simeq&\HomE (k , (\cE^{\vee})^{\vee})\\
                &\simeq&\Homk ( \cE^{\vee}\tensorE k ,k)\\
                &\simeq&\Homk (k \tensorR \kvee \tensorE k,k)\\
                &\simeq&\Homk (k \tensorR R^{\vee}  ,k)\\
                &\simeq&\HomR (k ,(R^{\vee})^{\vee})\\
                &\simeq&\HomR (k ,R)
\end{array}$$
\end{proof}





\section{Gorenstein duality for group cohomology}
\label{sec:GorHBG}
This section describes the key example. The first sign of this duality
was in Benson-Carlson duality \cite{BC1, BC2}, which in particular
shows that the Hilbert series of the group cohomology ring $H^*(BG)$
satisfies a functional equation if it is Cohen-Macaulay or a pair of
functional equations if it is almost Cohen-Macaulay. An algebraic
construction of the local cohomology spectral sequence was given
in \cite{groupca};  this was inspired by the  topological construction using equivariant
topology in \cite{KEG}, and a proof using structured equivariant
spectra first appears in 
\cite{Lieca}. The method described here comes from \cite{DGI}. 

\subsection{$p$-groups}
If $G$ is a $p$-group and $k$ is of characteristic $p$ we note that
$BG$ is $p$-complete and therefore $\Omega (BG_p^{\wedge})\simeq G$
and the Morita pair is 
$$R=C^*(BG;k) \mbox{ and } \cE =C_*(\Omega BG)=kG.$$
Since $kG$ is a Frobenius algebra, it is Gorenstein of shift 0, and so by Morita
invariance of the Gorenstein condition (Theorem
\ref{REGorenstein}),  $C^*(BG)$ is also Gorenstein
of shift 0. 

\subsection{General finite groups}
Now if $G$ is an arbitrary finite group, we may choose a faithful
 representation  $\rho : G\lra SU(n)$ for some $n$ and consider the fibration 
$$BSU(n) \lla BG \lla SU(n)/G .$$
Since $BSU(n)$ is simply connected, the Eilenberg-Moore theorem gives a
cofibre sequence
$$C^*(BSU(n))\lra C^*(BG)\lra C^*(SU(n)/G). $$
Now note that $H^*(BSU(n))$ is polynomial on $c_2, \ldots , c_n$ and therefore $C^*(BSU(n))$
is c-Gorenstein and therefore Gorenstein (with shift $2(2+3+\cdots
+n)-(n-1)=\dim (SU(n))$).  On the other hand $SU(n)/G$
is an orientable manifold of the same dimension as $SU(n)$  and therefore $C^*(SU(n)/G)$ is
Gorenstein. By  Gorenstein Ascent (Example \ref{eg:fibration}), $C^*(BG)$ is Gorenstein of shift 0 as
required. 

\subsection{The local cohomology theorem}
As described in Lemma \ref{lem:loccohss}
 we thus obtain the local cohomology theorem 
$$H^*_{\fm}(H^*BG)\Rightarrow H_*(BG)  $$
for group cohomology.

As described in Corollary \ref{cor:functeqn}, we then obtain
functional equations in many cases. We warn that $t$ is of codegree
$1$ (unlike $s$ in the corollary, which was of degree 1). 
If $H^*(BG)$ is Cohen-Macaulay, it is Gorenstein and its 
Hilbert series satisfies 
$$p(1/t)=(-1)^rt^{r}p(t).$$

If $H^*(BG)$ is almost Cohen-Macaulay it is also almost Gorenstein, and the
Hilbert series satisfies 
$$p(1/t)-(-1)^rt^{r}p(t)=(-1)^{r-1}(1+t)q(t) \mbox{ and } q(1/t)=(-1)^{r-1}t^{-r+1}q(t).$$ 

In any case $H^*(BG)$ is Gorenstein in codimension 0 and almost
Gorenstein in codimension 1. 

We will describe a number of 2-group examples with $k$ of
characteristic $2$.

\subsection{The elementary abelian group of rank $r$}
The group of order 2 may be viewed as a subgroup of the non-zero real
numbers. It therefore acts diagonally on $n\R =\R \oplus \cdots \oplus
\R$ and freely on the unit
sphere $S(n\R)$. Hence 
$$BC_2 =S(\infty \R)/C_2=\R P^{\infty}. $$ 
 
Similarly, if $G\cong C_2\times \cdots \times C_2$  is of rank $r$ we see that 
$$BG=\R P^{\infty}\times \cdots \times \R P^{\infty}, $$
and 
$$H^*(BG)=k[x_1, \ldots , x_r].$$
This is visibly of dimension $r$ and depth
$r$ and  
$$p(t)=\frac{1}{(1-t)^r}. $$
The functional equation is easily checked, and the reader may wish to
check directly that the local cohomology of $H^*(BG)$ is $H_*(BG)$ (up
to a shift).

\subsection{The quaternion group of order 8}
If $G=Q_8$ is quaternion of order 8 we note that it acts freely on the
unit 3-sphere in the quaternion algebra $S(\bH) $, and in fact 
$H^*(S(\bH)/G)=k[x,y]/(x^3,x^2+xy+y^2,
y^3)$. It is then clear that $G$ acts freely on the contractible space
$S(\infty \bH)$ and that 
$$H^*(BG)=H^*(S(\bH )/G)[z]$$
where $z$ is of codegree 4. This is visibly of dimension 1 and depth
1 and  
$$p(t)=\frac{1+2t+2t^2+t^3}{1-t^4}. $$
The functional equation is easily checked.  It is also easy to check
directly that the local cohomology of $H^*(BG)$ is $H_*(BG)$ (up to a shift).

Although there are many other groups for which we could calculate the
cohomology it is convenient to simply refer to the invaluable Jena
database \cite{Jena}  for the cohomology of small  $p$-groups. 

\subsection{The dihedral group of order 8}
If $G=D_8$ is dihedral of order 8 
$$H^*(BG)=k[x_1,y_1,z_2]/(xy). $$
It is quick to check that this is of  dimension  2  and depth
2 and  
$$p(t)=\frac{1}{(1-t)^2}. $$
The functional equation is easily checked. It is an easy exercise to check
 directly that the local cohomology of $H^*(BG)$ is  $H_*(BG)$ (up to
 a shift).

\subsection{The semi-dihedral group of order 16}
If $G=SD_{16}$ is the semi-dihedral group of order 16 
$$H^*(BG)=k[x_1,y_1,z_3,t_4]/(xy,x^3, xz,z^2+ty^2). $$
One may check that this is of  dimension  2  and depth
1 and is included as the first example which is not
Cohen-Macaulay. Its Hilbert series is 
$$p(t)=\frac{1}{(1-t)^2(1+t^2)},  $$
and one may check that the almost Cohen-Macaulay functional equations are
satisfied with shift 0.  It is presumably a coincidence that it also
satisfies the functional equation of a Cohen-Macaulay graded ring of
dimension 2 with (homological) shift $2$.

\subsection{Group Number 7 of order 32}
The 7th group in the Small Groups library list of the 51 groups of
order 32 is the only one whose cohomology is neither Cohen-Macaulay nor 
almost Cohen-Macaulay. The ring has a minimal presentation with 8
generators and 18 relations, so it won't be recorded here. 
The important facts for us are that it is of dimension  3  and depth 1 and has Hilbert series 
$$p(t)=\frac{1-t+t^2}{(1-t)^3(1+t^2)}. $$
By coincidence this satisfies the functional
equation for a Gorenstein graded ring with shift 0. 

Taking the description of its cohomology given in \cite{Rusin} one may
easily calculate its local cohomology. In fact it has a polynomial
subgring $P=k[z_1,x_2,s_4]$ over which it is a direct sum of 5
submodules, namely $P\oplus \Sigma_3P\oplus \Sigma_4 P\oplus M\oplus N$
where 
$$M=\cok (\Sigma_4P\stackrel{\{ x, z\} }\lra \Sigma_2P\oplus \Sigma_3P)$$
and 
$$N=\Sigma_1 P/(x,z). $$
The free submodule is of depth 3, $M$ is of depth 2. Finally
$$H^1_{\fm}(H^*(BG))=H^1_{\fm}(N)$$ 
is of dimension 1 over $k$ in degrees $3, 7, 11, 15, \cdots$.
Knowing the dimensions of $H_n(BG)$ from $p(t)$ we see   that the differential
$$d_2:H^1_{\fm}(H^*(BG))\lra \Sigma^{-1}H^3_{\fm}(H^*BG))$$ 
is a monomorphism (and so non-zero in infinitely many degrees).

\subsection{Other classes of groups}
We have described the fact that $C^*(BG)$ is Gorenstein for finite
groups $G$. It is shown in \cite{DGI} that $C^*(BG)$ is Gorenstein of
shift equal to $\dim (G)$ whenever $G$ is a compact Lie group,
provided that either $k$ is of characteristic 2 or the group of components is
of odd order. In  general there is a twisting by the representation
$H^g(G)$, which  is trivial in the aforementioned cases.

If $G$ is a virtual duality group of dimension $n$ there is a form of
Gorenstein duality with shift $-n$ for $C^*(BG)$. Now there is a 
 twisting by $H^n(G;kG)$, which is usually of infinite
 dimension. The algebraic proof is given in \cite{vdca}, along with a
 topological proof for arithmetic groups using equivariant topology. 
A proof along the present lines can be given for these groups as
follows. We choose a normal subgroup $N$ of finite index in $G$ which is a
duality  group, and a contractible space $X$ on which $G$ acts with
finite isotropy. Now $N$ acts freely on $X$ so $X/N$ is a manifold
with boundary. Now let $Q=G/N$, and consider the fibration 
$$BQ\lla   BG\lla X/N,  $$
which we may obtain from the equivalence 
$$BG\simeq EQ\times_Q (X/N). $$ 
We attempt to apply Gorenstein Ascent. In the situation that $X/N$
has empty boundary (i.e., $G$ is a virtual {\em Poincar\'e} duality
group) we infer $C^*(BG)$ is Gorenstein as required.

\section{Gorenstein duality for rational spaces}
\label{sec:GorQ}
This section discusses Gorenstein duality for rational spaces.
F\'elix-Halperin-Thomas \cite{FHT} have considered  the Gorenstein
condition in depth, so the distinctive feature here is the
concentration on  Gorenstein duality as in \cite{DGI}. 

For spaces with finite dimensional cohomology $X$ is Gorenstein if and only if $H^*(X)$ is Gorenstein,
but in general there are Gorenstein spaces for which $H^*(X)$ is not
Gorenstein and we make explicit the local cohomology theorem and its
consequences for  $H^*(X)$. 

\subsection{Fundamentals}
We specialize some of the above results to $C^*(X;\Q)$, beginning with
the zero dimensional case. 

\begin{lemma}\cite[3.6]{FHT}
If $H^*(A)$ is finite dimensional then 
$A$ is  Gorenstein if and only if $H^*(A)$ is a Poincar\'e duality 
algebra. 
\end{lemma}

\begin{proof}
If $H^*(A)$ is a Poincar\'e duality algebra of formal dimension 
$n$ then it is a zero dimensional Gorenstein ring with $a$-invariant
$-n$, so $A$ is Gorenstein with shift $-n$ by the previous corollary.

Conversely, if $A$ is Gorenstein of shift $a$, we have a Gorenstein 
duality spectral sequence. Since $H^*(A)$ is finite dimensional,  it 
is all torsion. Accordingly, $\lc^*(H^*(A))=H^*(A)$, 
and the spectral sequence reads
$$H^*(A)=\Sigma^a H^*(A)^{\vee}$$
and $H^*(A)$ is a Poincar\'e duality algebra of formal dimension $-a$.
\end{proof}

By applying Gorenstein Ascent  (Example \ref{eg:fibration})
one may use these to construct other examples which are Gorenstein
but not c-Gorenstein. 

\begin{prop} \cite[4.3]{FHT}
Suppose  we have a fibration $F\lra E \lra B$ with 
$F$ finite. If $F$ and $B$ are Gorenstein with 
shifts $f$ and $b$ then $E $ is Gorenstein with 
shift $e=f+b$.\qqed 
\end{prop}

\begin{lemma}
If $X$ is a simply connected rational space with $H^*(X)$ Noetherian
then if $C^*(X;\Q)$ is Gorenstein of shift $a$, then it is automatically orientable and so 
has Gorenstein duality of shift $a$ 
\end{lemma}

\begin{proof}
We see that $\Q$ is proxy-small by Corollary
\ref{cor:proxysmallegs}. We may now apply  the same argument as for Proposition
\ref{prop:autoGor} to see there is a unique action of $C_*(\Omega X;
\Q)$ on $\Q$ and deduce automatic Gorenstein duality. 
\end{proof}

\subsection{Examples}
The basic results of the previous subsection  allow us to construct innumerable examples. For example
any finite Postnikov system is Gorenstein \cite[3.4]{FHT}, 
so that in particular any sci space is Gorenstein.
A simple example will illustrate the duality.

\begin{example} 
\label{eg:hGornotGor}
We construct a rational space $X$ in a fibration 
$$S^3\times S^3 \lra X \lra \CP^{\infty} \times \CP^{\infty},  $$
so that $X$ is Gorenstein. We will calculate $H^*(X)$ and
observe that it is not Gorenstein.

Let $V$ be a graded vector space with two generators
$u,v$ in degree 2, and let $W$ be a graded vector space with two 
generators in degree 4. The two 4-dimensional cohomology classes
 $u^2, uv$ in $H^*(KV)=\Q[u,v]$ define a map $KV\lra KW$, 
and we let $X$ be the fibre, so we have a fibration
$$S^3\times S^3 \lra X \lra KV$$
as required.  By \cite{DGI}, this is Gorenstein
with shift $-4$ (being the sum of the shift (viz $-6$) of 
$S^3\times S^3$ and the shift (viz 2) of $KV$). 

It is amusing to calculate the cohomology ring of
$X$. It is $\Q[u,v,p]/(u^2,uv,up, p^2)$ where $u,v$ and $p$ 
have degrees $2,2$ and $5$. The dimensions of its graded components
are $1,0,2,0,1,1,1,1,1,\ldots$  (i.e., its Hilbert series 
is $p_X(t)=(1+t^5)/(1-t^2)  +t^2$, where $t$ is of codegree 1). 

In calculating local cohomology it is useful to note that 
${\frak{m}}=\sqrt{(v)}$.
The local cohomology is $\lc^0(H^*(X))=\Sigma_2\Q$ in degree 0
(so that $H^*(X)$ is not Cohen-Macaulay)
and as a $\Q[v]$-module $\lc^1(H^*(X))$ is 
$\Q[v]^{\vee}\tensor (\Sigma^{-3} \Q \oplus \Sigma^2 \Q)$.
Since there is no higher local cohomology the local cohomology spectral 
sequence necessarily collapses, and the resulting
exact sequence
$$0\lra \lc^1(H^*(X))\lra \Sigma^{-4}H^*(A)^{\vee} \lra \Sigma^{-2}\Q\lra 0$$
is consistent. 

Since the Cohen-Macaulay defect here is 1, we have a pair of functional 
equations
$$p_X(1/t)-(-t)t^{-4}p_X(t)=(1+t)\delta(t)$$
and 
$$\delta (1/t)=t^4\delta (t).$$
Indeed, the first equation  gives $\delta (t)=t^{-2}$, which is indeed
the Hilbert series of $\lc^0(H^*(X))^{\vee}$, and it obviously satisfies
the second equation.
\end{example}

\section{The ubiquity of Gorenstein ring spectra}
\label{sec:Goregs}
We have  considered  a number of examples, in some cases giving rather
complete proofs. This short section points out that the
ideas can be extended rather easily. In fact there is a sense in which
all the examples come from the first three by using Gorenstein Ascent
and Morita invariance of the Gorenstein condition. 

\begin{itemize}
\item Gorenstein commutative local rings $R$ (shift
  $-\mathrm{Krulldim}(R)$).  
\item $C^*(M)$ for closed manifolds $M$ (orientable if $M$ is orientable; shift $-\dim (M)$).
\item $C_*(G)$ for compact Lie groups $G$  (orientable if $G$ acts
  trivially on $H_g(G)$; shift $\dim (G)$).
\item $C^*(BG)$ for compact Lie groups $G$ (orientable if $G$ acts
 trivially on $H_g(G)$; shift $\dim (G)$).
\item $C^*(BG)$ for virtual duality groups $G$ (dualizing module $H_*(G;kG)$; shift $-\mbox{formal-dim} (G)$).
\item $C^*(EG\times_GM)$ Borel construction on a $G$-manifold $M^m$ for $G^g$
  compact Lie (orientable if $G$, $M$ and the action are; shift
  $g-m$). An amusing instance of this arises from toric geometry; if
  $K$ is a simplicial complex one may construct the so-called  moment
  angle  complex $Z_K$ which has an action of a torus $G$ whose Borel cohomology
  is the Stanley-Reisner ring $k[K]$. If $K$ is a simplicial sphere
  $Z_K$ is a manifold and $k[K]$ is Gorenstein.  See \cite{BP} for
  definitions and proofs.  
\item Chromatic examples. The coefficient ring $R$ (or ring spectrum) is contravariant in the
  geometric object making it like cochains, but the ring spectrum $R$
  is often  connective (unlike cochains). With this caveat, the examples are  parallel
  to the $C^*(BG)$ example with the compact Lie group $G$ replaced by
  an algebraic group $\G$, so that $R$ plays the role of the ring spectrum of functions on
  $\G$. The Morita counterpart of $R=\cO_{\G}$ would be the
   ring of operations $C_*(\G; k)=\Hom_R(k,k)$.   
\item Hochschild homology. In  \cite{THHGor} several classes of
  examples are identified where the  Hochschild homology with field
  coefficients inherits Gorenstein properties.  The context is that
we are given maps $S\lra R\lra k$ of commutative rings, so that we can define the
Hochschild homology 
$$HH_{\bullet}(R|S;k)=R\tensor_{R\tensor_SR}k, $$
and it is a ring spectrum with a map to $k$. The results then say that 
(under substantial additional hypotheses)
 if $R$ is Gorenstein of shift $a$ and $HH_{\bullet}(k|S;k)$  
is  Gorenstein of shift $b$ then $HH_{\bullet}(R|S;k)$ is Gorenstein of shift
$b-a$. The proof comes  out of Gorenstein Ascent, with the hypotheses designed to
allow the application of Lemma \ref{lem:Gorascent}. 
\end{itemize}

\appendix
\renewcommand{\thesubsection}{\Roman{subsection}}

\section{Algebraic definitions: Local and \v{C}ech cohomology and homology} \label{sec:loccoh}
Related surveys are given in  \cite{Handbook1, Handbook2}.
The material in this section is based on \cite{G, GML,Tateca}. Background in commutative
algebra can be found in \cite{Matsumura,BrunsHerzog}.

\subsection{The functors}
Suppose to begin with that $\AR $ is a commutative Noetherian ring and that 
$I=(\alpha_1, \dots , \alpha_n)$ is an ideal in $\AR $. 
We shall be concerned especially with two naturally occurring functors on 
$\AR $-modules: the $I$-power torsion functor and the $I$-adic completion functor.

The $I$-power torsion functor $\Gamma_I$ is defined by 
$$M \longmapsto \Gamma_I (M) = \{ x \in M \mid I^kx=0
\mbox{ for  } k >> 0  \}.$$
We say that $M$ is an {\em $I$-power torsion module} if $M=\Gamma_I M$. 
It is easy to check that the functor $\Gamma_I$ is left exact.

The $I$-adic completion functor is defined by 
$$M \longmapsto M_I^{\sm } = \ilim _k M/I^kM.$$
The Artin-Rees lemma implies that $I$-adic completion is exact on finitely generated modules, but it is 
neither right nor left exact in general.

\subsection{The stable Koszul complex.}
We begin with a sequence $\sset{\alpha_1,\ldots,\alpha_n}$ of elements of $\AR $
and define various chain complexes. In Subsection \ref{subsec:invariance}
we explain why the chain complexes only depend on the radical of the ideal
$I=(\alpha_1, \ldots, \alpha_n)$ generated by the sequence, in Subsection 
\ref{subsec:lochomcoh} we define associated homology groups,  
and in Subsection \ref{subsec:derived} we give conceptual 
interpretations of this homology under Noetherian hypotheses. 

We begin with a single element $\alpha \in \AR $, and an integer $s \geq 0$, and
define the $s$th unstable Koszul complex by 
$$K_s^{\bullet}(\alpha )=(\alpha^s :\AR  \darrow  \AR  )$$
where the non-zero modules are in cohomological degrees 0 and 1. 
These complexes form a direct system as $s$ varies,
$$\begin{array}{rcccc}
K_1^{\bullet}(\alpha)&=&
\left( \right. \;\; \AR & \stackrel{\alpha}\lra &\AR \left. \; \; \right)\\
\downarrow \hspace*{4ex}&&=\downarrow && \downarrow \alpha\\
K_2^{\bullet}(\alpha)&=&
\left( \right.\;\;\AR & \stackrel{\alpha^2}\lra &\AR \left. \; \; \right)\\
\downarrow \hspace*{4ex}&&=\downarrow && \downarrow \alpha\\
K_3^{\bullet}(\alpha)&=&
\left( \right. \; \; \AR & \stackrel{\alpha^3}\lra &\AR \left. \; \;\right)\\
\downarrow \hspace*{4ex} &&=\downarrow && \downarrow \alpha
\end{array}$$
 and the direct limit is  the {\em flat stable} Koszul complex 
$$\FK{\alpha} = \left( \AR  \darrow  \AR [1/\alpha ] \right). $$ 

When defining local cohomology, it is usual to use the complex $\FK{\alpha}$ 
of flat modules. However, we shall need a complex of projective $\AR$-modules 
to define the dual local homology modules. Accordingly, we take a 
particularly convenient projective approximation $\PPK{\alpha}$ to $\FK{\alpha}$. 
Instead of taking the direct limit of the $K^{\bullet}_s(\alpha )$, we take their 
homotopy direct limit. This makes the translation to the  topological context
straightforward. More concretely, our model for $\PPK{\alpha}$ is
displayed as the upper row in the homology isomorphism 
$$\diagram PK^{\bullet}(\alpha)  \dto&=& \left( \right. 
\AR  \oplus \AR [x] \dto_{\langle 1,0 \rangle} 
\rrto^<(0.35){\langle 1,\alpha x-1\rangle} & & 
\AR [x] \dto^{g} \left.\right)\\ 
K^{\bullet}_{\infty}(\alpha) &=& \left( \right. \hspace*{3ex}\AR \hspace*{3ex} \rrto & &\AR [1/\alpha] \left. \right), \\ 
\enddiagram$$ 
where $g(x^i)=1/\alpha^i$. Like $\FK{\alpha}$, this choice of $\PPK{\alpha}$ is 
non-zero only in cohomological degrees $0$ and $1$.

The stable Koszul cochain complex for a sequence $\AL = (\alpha_1, \dots , \alpha_n ) $ 
is obtained by tensoring together the complexes for the  elements, so that 
$$ \FK{\AL} =  \FK{\alpha_1} \otimes_R \dots \otimes_R \FK{\alpha_n},$$ 
and similarly for the projective complex $\PPK{\AL}$.

\subsection{Invariance statements.}
\label{subsec:invariance}

We list some basic properties of the stable Koszul complex, leaving proofs to the reader. 

\begin{lemma}\label{Ksupp}
If $\beta$ is in the ideal $I=(\alpha_1 , \alpha_2, \ldots , \alpha_n )$, then 
$\FK{\AL }[1/\beta ]$ is exact.
\end{lemma}

Note that, by construction, we have an augmentation map
$$\epz : \FK{\AL } \darrow  \AR .$$
Using this to compare different stable Koszul complexes and Lemma
\ref{Ksupp} to see when they are equivalences we deduce an important
invariance statement. 

\begin{cor} \label{indepgens}
Up to quasi-isomorphism, the complex $\FK{\AL}$ depends only on the radical of the ideal $I$. 
\end{cor}

In view of Corollary \ref{indepgens} it is reasonable to  write $\FK{I}$ for $\FK{\AL}$. 
 Since $\PPK{\AL}$ is a projective approximation to $\FK{\AL}$, 
it too depends only on the radical of $I$. We also write 
$K_s^{\bullet}(I) = K_s^{\bullet}(\alpha_1) \otimes \cdots \otimes K_s^{\bullet}(\alpha_n )$, but 
this is an abuse of notation since even its homology groups do depend on the choice of generators. 

\subsection{Local homology and cohomology.}
\label{subsec:lochomcoh}
The local cohomology and homology of an $\AR $-module $M$ are then defined by 
$$ H^*_I(\AR ;M) = H^*(\PPK{I} \otimes M) $$ and
$$   H_*^I(\AR ;M) = H_*(\Hom (\PPK{I} , M). $$
Note that we could equally well use the flat stable Koszul complex in the definition of local 
cohomology, as is more usual. Lemma \ref{Ksupp}
 shows that $H^*_I(M)[1/\beta]=0$ if $\beta \in I$, so $H^*_I(M)$ 
is an $I$-power torsion module and  supported over $I$.

It is immediate from the definitions that local cohomology and local homology are related by 
 a third quadrant universal coefficient spectral sequence 
\begin{equation} 
E_2^{s,t}= \mbox{Ext}^s_\AR (H_I^{-t}(\AR ),M) \Longrightarrow  H_{-t-s}^I(\AR ;M), 
\end{equation} 
with differentials $d_r:E_r^{s,t} \darrow  E_r^{s+r,t-r+1}$. 

\subsection{Derived functors}
\label{subsec:derived}

We gave our definitions in terms of specific chain complexes. 
The meaning of the definitions appears in the following two 
theorems.

\begin{thm}[Grothendieck \cite{G}]
\label{loccohR}
If $\AR $ is Noetherian, then the local cohomology
groups calculate the right derived functors of the left exact functor 
$M \longmapsto \Gamma_I (M)$. In symbols, 
$$\hspace{6mm} H_I^n(\AR ;M) = (R^n\Gamma_I)(M). \qqed$$ 
\end{thm}

This result may be used to give an explicit expression for 
local cohomology in familiar terms.  Indeed, 
since $\Gamma_I (M)=\colim_r\, \Hom(\AR /I^r,M)$, and the right
derived functors of the right-hand side are obvious, we
have
$$(R^n\Gamma_I)(M)\iso \colim_r\,\Ext_\AR^n(\AR /I^r,M).$$
The description in terms of the stable Koszul complex is usually 
more practical.

\begin{thm}[Greenlees-May \cite{GML}]
\label{lochomL}
If $\AR $ is Noetherian, then the local homology
groups calculate the left derived functors of the (not usually right exact) 
$I$-adic completion functor $M \longmapsto M_I^{\sm }$. Writing $L_n^I$ for
the left derived functors of $I$-adic completion, this gives 
$$ \hspace{6mm} H_n^I(\AR ;M) = L_n^I(M). \qqed$$ 
\end{thm}

The conclusions of Theorem \ref{loccohR} and  \ref{lochomL} 
are true under much weaker hypotheses \cite{GML,Schenzel}.

\subsection{The shape of local cohomology.}
One is used to the idea that $I$-adic completion is often exact, so that $L^I_0$ is the 
most significant of the left derived functors. However, it is the top non-vanishing 
right derived functor of $\Gamma_I$ that is the 
most significant. Some idea of the shape of these derived functors can be obtained 
from the following result. Observe that the complex  $\PPK{\AL }$ is non-zero only in 
cohomological degrees between $0$ and $n$, so that local homology and 
cohomology are zero above dimension $n$. A result of Grothendieck's usually gives 
a much better bound. We write $\dim (R)$ for the Krull dimension of $R$ and
 $\depth_I(M)$ for the $I$-depth of a module $M$
(the length of the longest $M$-regular sequence from $I$).

\begin{thm}[Grothendieck \cite{Go}]
\label{loccohvan}
If $\AR $ is Noetherian of Krull dimension $d$, then 
$$H^i_I(M)=0 \ \  \tand \ \ H_i^I(M)=0 \ \ \text{if} \ i>d.$$ 
If $e=\depth_I(M)$ then 
$$H^i_I(M)=0 \ \text{if} \ i< e.$$
If $\AR $ is Noetherian, $M$ is finitely generated, and $IM\neq M$, then 
$$H^e_I(M)\neq 0. \qqed$$ 
\end{thm}

Grothendieck's proof of  vanishing begins by noting that
 local cohomology is sheaf cohomology with support. It
then proceeds by induction on the Krull dimension and 
reduction to the irreducible case. The statement about depth is 
elementary, and proved  by induction on the length of the $I$-sequence
(see \cite[16.8]{Matsumura}).

The Universal Coefficient Theorem gives a useful consequence for local 
homology.

\begin{cor}\label{cmcase} 
If $\AR $ is Noetherian and $\mathrm{depth}_I(\AR )=\mathrm{dim}(\AR )=d$, then 
$$\hspace{6mm} L^I_sM=\mathrm{Ext}^{d-s}_\AR (H^d_I(\AR ), M). \qqed $$ 
\end{cor}

For example if $\AR =\Z$ and $I=(p)$, then $H^*_{(p)}(\Z )=H^1_{(p)}(\Z )=\Z /p^{\infty}$. 
Therefore the corollary states that $$L_0^{(p)}M=\mathrm{Ext}_{\Z}(\Z /p^{\infty} , M) \ \ \ \mbox{ and } \ \ \ 
L_1^{(p)}M=\Hom_{\Z}(\Z /p^{\infty} , M), $$
as was observed in Bousfield-Kan \cite[VI.2.1]{BK}.

\subsection{\v{C}ech homology and cohomology.}
\label{subsec:Cech}
We have motivated local cohomology in terms of $I$-power torsion, 
and it is natural to consider the difference between the torsion 
and the original module.  
In geometry this difference would be more fundamental than the torsion 
itself, and local cohomology would then arise by considering functions with support. 

To construct a good model for this difference, 
observe that $\epz: \FK{\AL } \darrow  \AR  $ is an 
isomorphism in degree zero and define the flat \v{C}ech complex $\FC{I}$ to be the complex 
$\SI(\ker\epz)$. Thus, if $i\geq 0$, then $\check{C}^i(I)=K^{i+1}(I)$. For example, if 
$I=(\alpha , \beta )$, then
$$ \FC{I}=\left( \: \AR [1/\alpha ] \oplus \AR [1/\beta ] \darrow  
\AR [1/( \alpha \beta )]\;\right) .$$ 
The differential $K^0(I)\darrow K^1(I)$ specifies a chain map $\AR \darrow \FC{I}$ whose 
fibre is exactly $\FK{I}$. Thus we have a fibre sequence 
$$ \fbox{$\FK{I}  \darrow  \AR  \darrow  \FC{I} .$}$$ 
We define the projective version $\PC{I}$ similarly, using the kernel of the composite of 
$\epz$ and the quasi-isomorphism $\PPK{I}\darrow\FK{I}$; note that $\PC{I}$ is non-zero in 
cohomological degree $-1$.

The \v{C}ech cohomology of an $\AR $-module $M$ is then defined by 
$$ \check{C}H^*_I(\AR ;M) = H^*(\FC{I} \otimes M) .$$

\subsection{\v{Cech} cohomology and \v{C}ech covers.}
To explain why $\FC{I}$ is called the \v{C}ech complex, we describe how it 
arises by using the \v{C}ech construction to calculate cohomology from a suitable open cover. 
More precisely, let $Y$ be the closed subscheme of $X=Spec(\AR )$ determined by $I$. The space 
$V(I)=\{ \wp | \wp\supset I\}$ decomposes as 
$V(I)=V(\alpha_1 ) \cap \ldots \cap V( \alpha_n )$, and there results an open cover of the 
open subscheme $X-Y$ as the union of the complements 
$X-Y_i$ of the closed subschemes $Y_i$ determined by the principal ideals $(\alpha_i)$. 
However, $X-Y_i$ is isomorphic to the affine scheme $Spec(\AR [1/\alpha_i ])$. Since affine 
schemes have no higher cohomology, 
$$H^*(Spec(\AR [1/\alpha_i ]);\tilde{M})=H^0(Spec(\AR [1/\alpha_i ]);\tilde{M})=M [1/\alpha_i ],$$
where $\tilde{M}$ is the sheaf associated to the $\AR$-module $M$. 
Thus the $E_1$ term of the Mayer-Vietoris spectral sequence for this cover 
collapses to the chain  complex $\FC{I}$, and
$$H^*(X-Y;\tilde{M})\iso \check{C}H^*_I(M).$$

\section{Spectral analogues of the algebraic definitions} 
\label{sectoploccoh}

We now transpose the algebra from Appendix \ref{sec:loccoh} into the
context of spectra.
It is convenient to note that it is routine to extend the algebra to graded rings, 
and we will use this without further comment below. 
We assume the reader is already comfortable working with ring spectra, but there is
an introduction in Sections \ref{sec:ringspectra1} and \ref{sec:constructspectra}, and one at greater length in \cite{spectra}. 

We replace the standing assumption that $\AR$ is a
 commutative $\Z$-algebra  by the assumption that it is 
a commutative $\bbS$-algebra, where $\bbS$ is the sphere 
spectrum. The category of $\AR $-modules is now   the category of 
$R$-module {\em spectra}. Since the derived category of a ring $R$ is equivalent to 
the derived category of the associated Eilenberg-MacLane spectrum \cite{Shipley}, 
the work of Appendix \ref{sec:loccoh}  can be reinterpreted in the new context. 
To emphasize the algebraic analogy, 
 we write $\tensor_R$ and $\Hom_R$ for the smash product over $R$ and function spectrum 
of $R$-maps  and $0$ for the trivial module. In particular $X\tensor_{\bbS}Y=X\sm Y$ 
and $\Hom_{\bbS}(X,Y)=F(X,Y)$.

This section is based on \cite{KEG,Tateca}.

\subsection{Koszul spectra.}
For $\alpha \in \pi_{*} R$, we define the stable 
Koszul spectrum $K_{\infty}(\alpha)$ by the fibre sequence 
$$K_{\infty}(\alpha)\longrightarrow R \longrightarrow R [1/\alpha], $$ 
where 
$R[1/ \alpha] = \hocolim (R \stackrel{\alpha}{\longrightarrow} 
R \stackrel{\alpha}{\longrightarrow} \ldots)$. Analogous to the filtration 
by degree in chain complexes, we obtain a filtration of the 
$R$-module $K_{\infty}(\alpha)$ by viewing it as 
$$\Sigma^{-1}(R[1/\alpha] \cup CR).$$

Next we define the stable Koszul spectrum for the sequence $\alpha_1, \ldots , \alpha_n$ by 
$$K_{\infty}(\alpha_1, \ldots , \alpha_n) = K_{\infty}(\alpha_1) \tensorR \cdots \tensorR K_{\infty}(\alpha_n),$$
and give it the tensor product filtration.

The topological analogue of Lemma \ref{Ksupp} states that if $\beta \in I$ then 
$$K_{\infty}(\alpha_1, \ldots , \alpha_n)[1/\beta ] \simeq 0;$$ 
this follows from Lemma \ref{Ksupp} and the spectral sequence (\ref{algtopss}) below.
We may now use precisely the same proof as in the algebraic case to conclude that the homotopy type of $K_{\infty}(\alpha_1, \ldots , \alpha_n)$ depends only on the radical of the ideal $I = (\alpha_1, \cdots , \alpha_n)$. We therefore write $K_{\infty}(I)$ for 
$K_{\infty}(\alpha_1, \ldots , \alpha_n)$.

\subsection{Localization and completion.}
\label{subsec:loccomp}
With motivation from Theorems \ref{loccohR} and \ref{lochomL}, we define the homotopical 
$I$-power torsion 
(or local cohomology) and homotopical completion (or local homology) modules associated 
to an $R$-module $M$ by
\begin{equation}
\Gamma_I(M)= K_{\infty}(I)\tensorR M \ \ \ \tand \ \ \ \Lambda_I(M)=M^{\wedge}_I=\HomR(K_{\infty}(I), M) . \end{equation} 
In particular, $\Gamma_I(R)=K_{\infty}(I)$.

Because the construction follows the algebra so precisely, it is easy give methods of calculation for the homotopy groups of these $R$-modules. We use the product of the filtrations of the $K_{\infty}(\alpha_i)$ given 
above and obtain spectral sequences
\begin{equation} \label{algtopss}
E^2_{s,t} = H^{-s,-t}_{I} (R_{* }; M_{* }) \Rightarrow \pi_{s+t} (\Gamma_IM) \end{equation} with differentials $d^r: E^r_{s,t} \rightarrow E^r_{s-r,t+r-1}$ and 
\begin{equation}
E_2^{s,t} = H_{-s,-t}^{I} (R^{* }; M^{* }) \Rightarrow \pi_{-(s+t)} (M^{\wedge}_I) \end{equation} with differentials $d_r: E_r^{s,t} \rightarrow E_r^{s+r,t-r+1}.$

\subsection{The \v{C}ech spectra.}
Similarly, we define the \v{C}ech spectrum by the cofibre sequence 
\begin{equation}\label{topKAC}
\fbox{$K_{\infty}(I) \darrow  R \darrow  \check{C}(I). $}
\end{equation}

\noindent We define the homotopical localization (or \v{C}ech 
cohomology) and \v{C}ech homology modules 
associated to an $R$-module $M$ by
\begin{equation}
M[I^{-1}]= \check{C}(I) \tensorR M \ \ \ \mbox{ and }\ \ \ 
 \Delta^I(M)=\HomR(\check{C}(I), M).
\end{equation}
In particular, $R[I^{-1}]= \check{C}(I)$. Once again, we have spectral sequences for 
calculating their homotopy groups from the analogous algebraic constructions.

\subsection{Basic properties.}
We can now give topological analogues of some basic pieces of algebra that we used in 
Section \ref{sec:loccoh}.  Recall that the algebraic Koszul complex $\FK{I}$ is a direct limit of 
unstable complexes $K_s^{\bullet}(I)$ that are finite complexes of free modules with 
homology annihilated by a power of $I$. We say 
that an $R$-module $M$ is a $I$-power torsion module if its $R_*$-module $M_*$ of homotopy groups 
is a $I$-power torsion module; equivalently, $M_*$ must have support over $I$. 

\begin{lemma} \label{HJcolim}
The $R$-module $\HFK{I}$ is a homotopy direct limit of finite $R$-modules 
$K_s(I)$, each of which has homotopy groups annihilated by some power of $I$. Therefore 
$\HFK{I}$ is a $I$-power torsion module.
\end{lemma}

The following lemma is an analogue of the fact that $\FC{I}$ is a chain complex which 
is a finite sum of modules $R[1/\alpha ]$ for $\alpha \in I$. 

\begin{lemma} \label{CJfilt}
The $R$-module $\CJI$ has a finite filtration by $R$-submodules with subquotients that are suspensions of modules of the form $R[1/\alpha ]$ 
with $\alpha \in I$. 
\end{lemma}

These lemmas are useful in  combination.

\begin{cor} \label{comp}
If $M$ is a $I$-power torsion module
then $M \tensorR \CJI \simeq 0$; in particular $K_{\infty}(I) \tensorR \CJI \simeq 0$. 
\end{cor} 
\begin{proof} Since $M[1/\alpha ] \simeq 0$ for $\alpha \in I$, Lemma \ref{CJfilt} gives the 
conclusion for $M$. 
\end{proof}

\section{Completion at ideals and Bousfield localization} \label{seccompletionssht}

Bousfield localizations include both completions at ideals and localizations at 
multiplicatively 
closed sets, but one may view these Bousfield localizations as falling into the types 
typified by completion at $p$ and localization away from $p$. Thinking in terms of 
$Spec(R_*)$, 
this is best viewed as the distinction between localization at a closed set 
and localization at the complementary open subset. In this section 
we deal  with the closed sets and  with the open sets in Section \ref{sec:locaway}. 
This appendix  is based on \cite{GM, GML,Tateca}.

\subsection{Homotopical completion.}
\label{htpicalcompletion}
As observed in the proof of Lemma \ref{HJcolim}, we have 
$K_{\infty}(\alpha )=\hocolim_s\Sigma^{-1} R/\alpha^s$ and therefore 
$$M^{\wedge}_{(\alpha)} = \HomR(\hocolim_s \Sigma^{-1} R/\alpha^s , M) 
\simeq \holim_s M/\alpha^s. $$ 
If $I=( \alpha, \beta )$, then 
$$\begin{array}{rcl}
 M^{\wedge}_I &=& \HomR(K_{\infty}(\alpha) \tensorR  K_{\infty}(\beta), M)\\
              &=&\HomR(K_{\infty}(\alpha ), \HomR( K_{\infty}(\beta ), M))\\
              &=&(M_{(\beta) }^{\wedge})_{(\alpha)}^{\wedge},
\end{array}$$
and so on inductively. This should help justify the notation $M_I^{\wedge}=\HomR(K_{\infty}(I),M).$

When $R=\bbS$ is the sphere spectrum and $p \in \Z \cong \pi_0(\bbS)$, $K_{\infty}(p)$ is a Moore 
spectrum for $\Z / p^{\infty}$ in degree $-1$ and we recover the usual definition 
$$X^{\wedge}_p=F(\Sigma^{-1}\bbS /p^{\infty}, X)$$ 
of $p$-completions of spectra as a special case, where $F(A,B)=\Hom_{\bbS}(A,B)$ is the 
function spectrum. The standard short exact sequence 
for the calculation of the homotopy groups of $X^{\wedge}_p$ in terms of `Ext completion' 
and `Hom completion' follows directly from Corollary \ref{cmcase}.

Since $p$-completion has long been understood to be an example of a Bousfield localization, 
our next task is to show that completion at $I$ is a Bousfield localization in general.

\subsection{Bousfield's terminology.}
 Fix an $R$-module  $E$. A spectrum  $A$ is {\em $E$-acyclic} if $A  \tensorR  E \simeq 0$; 
a map $f:X \darrow  Y$ is an { \em $E$-equivalence} if its cofibre is $E$-acyclic.  
An $R$-module $M$  is {\em $E$-local} if $E \tensorR T \simeq 0$ 
implies  $\HomR (T, M) \simeq 0$. A map $Y \darrow  L_EY$ is a 
{\em Bousfield $E$-localization} of  $Y$ if it is an $E$-equivalence 
and $L_EY$   is  $E$-local. This means that $Y \darrow  L_EY$ is  terminal
among $E$-equivalences with domain $Y$, and the Bousfield localization is therefore unique 
if it exists. Similarly, we may replace the single spectrum $E$ by a class $\sE$ 
of objects $E$, and require the conditions hold for all such $E$ 

The following is a specialization of a change  of rings result to the
ring map $\bbS \lra R$.

\begin{lemma} \label{loclem}
Let ${\sE}$ be a class of $R$-modules. If an $R$-module $N$ is 
${\sE}$-local as an $R$-module, then it is ${\sE}$-local as an $\bbS$-module. 
\end{lemma} 
\begin{proof} If $E\sm T =E\tensor_{\bbS} T \htp *$ for all $E$, 
then $E\tensorR (R \tensor_{\bbS} T)\htp 0$ for all $E$ and therefore 
$F(T,N)=\Hom_{\bbS}(T,N) \htp \HomR(R \tensor_{\bbS} T,N) \htp 0$. 
\end{proof}

\subsection{Homotopical completion is a Bousfield localization.}
The  class that will concern us most is the class $I$-{\bf Tors} of finite $I$-power torsion $R$-modules $M$. Thus $M$ must be a finite cell $R$-module, and its $R_*$-module $M_*$ of homotopy groups must be a $I$-power 
torsion module. 

\begin{thm} \label{Bousclosed}
For any finitely generated ideal $I$ of $R_*$ the map $M \darrow  M_I^{\wedge}$ is 
Bousfield localization in the category of $R$-modules in each of the following 
equivalent senses:
\begin{enumerate}
\item[(i)] with respect to the $R$-module $\Gamma_I(R)=K_{\infty}(I)$. \item[(ii)] with respect to the class $I$-{\bf Tors} of finite $I$-power torsion $R$-modules. 
\item[(iii)] with respect to the $R$-module $K_s(I)$ for any $s \geq 1$. 
\end{enumerate} Furthermore, the homotopy groups of the completion are related to local homology groups 
by a spectral sequence 
$$E^2_{s,t}=H^I_{s,t}(M_*)\Longrightarrow \pi_{s+t}(M_I^{\wedge}).$$ 
If $R_*$ is Noetherian, the $E^2$ term consists of the left derived functors of $I$-adic completion: $H^I_s(M_*)=L_s^I(M_*)$. \end{thm} 
\begin{proof}  We begin with (i). Since 
$$\HomR(T, M_I^{\wedge}) \simeq \HomR(T \tensorR K_{\infty}(I), M),$$
it is immediate that $M_I^{\wedge}$  is $K_{\infty}(I)$-local. We must prove that the map $M \darrow  M_I^{\wedge}$ is a $K_{\infty}(I)$-equivalence. The fibre of this map is $\Hom_R (\CJI , M)$, so we must show that  
$$\Hom_R (\CJI , M) \tensorR K_{\infty}(I) \simeq 0.$$
By Lemma \ref{HJcolim}, $K_{\infty}(I)$ is a homotopy direct limit of terms $K_s(I)$. Each $K_s(I)$ is in $I$-{\bf Tors}, and we see by their definition in terms of cofibre sequences and smash products that their duals 
$DK_s(I)$ are also in $I$-{\bf Tors}, where $DM=\HomR(M,R)$. 
Since $K_s(I)$ is a finite cell $R$-module,
$$\HomR(\CJI , M) \tensor_R K_s(I)= \HomR( \CJI \tensorR DK_s(I),M),$$ and $\CJI \tensorR DK_s(I) \simeq 0 $ by Corollary \ref{comp}. Parts (ii) and 
(iii) are similar but simpler. For (iii), observe that we have a cofibre sequence 
$R/\alpha^s \darrow  R/\alpha^{2s} \darrow  R/\alpha^s$, so that all of the $K_{js}(I)$ may be constructed from $K_s(I)$ using a finite number of cofibre sequences. 
\end{proof}

\section{Localization away from ideals and Bousfield localization}
\label{sec:locaway}

In this section we turn to localization away from the closed set defined
by an ideal $I$. First, observe that, when $I=(\alpha)$, 
$M[I^{-1}]$ is just $R[\alpha^{-1}]\tensor_R M = M[\alpha^{-1}]$. However, the higher 
\v{C}ech cohomology groups give the construction for general finitely generated ideals 
a quite different algebraic flavour, and $M[I^{-1}]$ is rarely a localization 
of $M$ at a multiplicatively closed subset of $R_*$. 
This appendix is based on \cite{Tateca}.

\subsection{The \v{C}ech complex as a Bousfield localization.}
To characterize this construction as a Bousfield 
localization, we consider the class $I$-{\bf Inv} of $R$-modules $M$ for which there is 
an element $\alpha \in I$ such that $\alpha :M \darrow M$ is an equivalence. 

\begin{thm}
For any finitely generated ideal $I=(\alpha_1,\ldots,\alpha_n)$ of $R_*$, the map $M \darrow M[I^{-1}]$ is Bousfield localization in the category of $R$-modules in each of the following equivalent senses: \begin{enumerate} \item[(i)] with respect to the $R$-module $R[I^{-1}]=\CJI$. \item[(ii)] with respect to the class  $I$-{\bf Inv}. \item[(iii)] with respect to the set $\{ R[1/\alpha_1 ], \ldots , R[1/\alpha_n] \}$. \end{enumerate} Furthermore, the homotopy groups of the localization are related to \v{C}ech cohomology 
groups by a spectral sequence
$$E_{s,t}^2 = \check{C}H_I^{-s,-t}(M_*) \Longrightarrow \pi_{s+t}(\CIM ).$$ If $R_*$ is Noetherian, the $E^2$ term can be viewed as the cohomology of
$Spec(R_*) \setminus V(I)$ with coefficients in the sheaf associated to $M_*$. 
\end{thm}

\begin{remark}
\label{rk:cellular}
One may also characterize the map $\Gamma_I (M) \lra M$ by a universal property
analogous to that of the cellular approximation in spaces: it is the
{\em $K_1(I)$-cellularization} of $M$.

Indeed, on the one hand, $\Gamma_I(M)$ is constructed from $K_1(I)$ by \ref{HJcolim},  
and on the other hand, the map induces an equivalence of $\HomR (K_1(I), \cdot )$
since, by Lemma \ref{CJfilt}, 
$\HomR(K_1(I), \CJI )\simeq 0.$
\end{remark}

\end{document}